\documentclass[a4paper,10pt,reqno]{amsart}

\usepackage[utf8]{inputenc}
\usepackage[T1]{fontenc}
\usepackage{amsthm}
\usepackage{amssymb}
\usepackage[inline]{enumitem}
\usepackage{caption}
\usepackage{color}
\usepackage{hyperref}
\usepackage{fancyhdr}
\usepackage{mathrsfs}
\usepackage{stmaryrd}
\usepackage{tikz-cd}

\theoremstyle{definition}

\newtheorem{theorem}{Theorem}[section]
\newtheorem{lemma}[theorem]{Lemma}
\newtheorem{proposition}[theorem]{Proposition}
\newtheorem{corollary}[theorem]{Corollary}

\theoremstyle{definition}
\newtheorem{definition}[theorem]{Definition}
\newtheorem{example}[theorem]{Example}
\newtheorem{examples}[theorem]{Examples}
\newtheorem{remark}[theorem]{Remark}
\newtheorem{remarks}[theorem]{Remarks}

\definecolor{blue-url}{RGB}{0,0,100}
\definecolor{red-url}{RGB}{100,0,0}
\definecolor{green-url}{RGB}{0,100,0}

\hypersetup{
	pdftitle={An Abstract Factorization Theorem and Some Applications},
	pdfauthor={Salvatore Tringali},
	pdfmenubar=false,
	pdffitwindow=true,
	pdfstartview=FitH,
	colorlinks=true,
	linkcolor=blue-url,
	citecolor=green-url,
	urlcolor=red-url
}

\renewcommand{\emptyset}{\varnothing}
\renewcommand{\setminus}{\smallsetminus}
\renewcommand{\,}{\kern 0.1em}

\providecommand\llb{\llbracket}
\providecommand\rrb{\rrbracket}
\providecommand{\RR}{\mathbin{R}}

\DeclareMathOperator{\myprod}{\Pi}
\DeclareMathOperator{\hgt}{ht}

\newcommand{\evid}[1]{\textsf{#1}}

\begin{document}
\title{An Abstract Factorization Theorem and Some Applications}
\author{Salvatore Tringali}
\address{School of Mathematical Sciences,
Hebei Normal University | Shijiazhuang, 050024 China}
\email{salvo.tringali@gmail.com}
\urladdr{\href{https://imsc.uni-graz.at/tringali}{https://imsc.uni-graz.at/tringali}}



\begin{abstract}
\noindent
We combine the language of monoids with the language of preorders so as to refine some fundamental aspects of the classical theory of factorization and prove an abstract factorization theorem with a variety of applications. In particular, 
we obtain a generalization, from cancellative to Dedekind-finite (commutative or non-commutative) monoids, of a classical theorem on ``atomic factorizations'' that traces back to the work of P.\,M.~Cohn in the 1960s;  recover a theorem of D.\,D.~Anderson and S.~Valdes-Leon on ``irreducible factorizations'' in commutative rings;  improve on a theorem of A.\,A.~Antoniou and the author that characterizes atomicity in certain ``monoids of sets'' naturally arising from additive number theory and arithmetic combinatorics; and give a monoid-theoretic proof that every module of finite uniform dimension over a (commutative or non-commutative) ring $R$ is a direct sum of finitely many indecomposable modules (this is in fact a special case of a more general decomposition theorem for the objects of certain categories with finite products, where the indecomposable $R$-modules are characterized
as the atoms of a suitable ``monoid of modules'').
\end{abstract}
\maketitle
\thispagestyle{empty}

\section{Introduction}
\label{sec:1}
Let $H$ be a monoid (see Sect.~\ref{sec:2.2} for notation and terminology). As usual, a principal right i\-de\-al of $H$ is a set of the form $aH$ (with $a \in H$); and we say that $H$ satisfies the as\-cend\-ing chain condition (ACC) on principal right ideals (ACCPR) if there is no infinite sequence of principal right ideals of $H$ that is (strictly) increasing with respect to inclusion. The ACC on principal left ideals (ACCPL) and the ACC on principal ideals (ACCP) are defined in a similar way, with principal right ideals replaced, resp., by sets of the form $Ha$ (that is, principal left ideals) and sets of the form $HaH$ (that is, principal ideals). 

The ACCPR, the ACCPL, and the ACCP (one and the same condition in the commutative set\-ting) have been the subject of extensive research and are known to play a critical role in the study of the ``arithmetic of monoids and rings''.
More in detail, let an \textsf{atom} of $H$ be a non-unit $a \in H$ such that $a \ne \allowbreak xy$ for all non-units $x, y \in H$; and an \evid{irreducible} of $H$ be a non-unit $a \in H$ such that $a \ne xy$ for all non-units $x, y \in H$ with $HxH \ne \allowbreak HaH \ne \allowbreak HyH$. It is a classical result, henceforth referred to as \emph{Cohn's theorem} for it traces back to P.\,M.~Cohn's work on factorization in the 1960s (see, in particular, Theorem 2.8 in \cite{Co64} and the unnumbered corollary on the bottom of p.~589 in \cite{Co67}), 
that every non-unit in a cancellative monoid satisfying the ACCPR and the ACCPL factors as a finite product of atoms (i.e., the monoid is \evid{atomic}), see \cite[Proposition 0.9.3]{Co06}:
An extension of this result to ``nearly cancellative'' monoids was recently obtained by Y.~Fan et al.~in \cite[Lemma 3.1(1)]{FaGeKaTr17} (the commutative case) and \cite[Theorem 2.28(i)]{Fa-Tr18}.
In a similar vein, it was ap\-par\-ent\-ly first observed by D.\,D.~Anderson and S.~Valdes-Leon in \cite[Theorem 3.2]{AnVL96} that every non-unit of a \emph{commutative} monoid satisfying the ACCP factors into a finite product of irreducibles (Anderson and Valdes-Leon state their theorem only for commutative rings, but the proof carries over verbatim to commutative monoids): 
This extends Cohn's theorem in a different direction than the one taken by Fan et al., since every atom is obviously an irreducible and, as a partial converse, every irreducible in a cancellative commutative monoid is an atom (see Remark \ref{rem:3.7}\ref{rem:3.7(2)} and Corollary \ref{cor:4.4} for a more comprehensive analysis of the relations between atoms and irreducibles).

On the whole, the above results can be regarded as a far-reaching generalization of the Fundamental Theorem of Arithmetic --- that every integer greater than one factors as a product of prime numbers (in an essentially unique way) --- and lie in the foundations of a subfield of algebra known as factorization theory \cite{GeHK06, GeZh19}. But ``factorization theorems'' are common to many other fields:
\begin{enumerate}[label=(\textup{\textsc{f}\arabic{*}})]
\item\label{it:F1} It is basic algebra (see, e.g.,~\cite[Proposition (19.20)]{La01})
that every artinian or noetherian $R$-module 
is an internal direct sum of finitely many indecomposable submodules,
where $R$ is a commutative or non-com\-mu\-ta\-tive ring. (Here as usual, an $R$-module $M$ is \evid{indecomposable} if $M$ is neither a zero $R$-module nor the direct sum of two non-zero $R$-modules.)
\item\label{it:F2} It is folklore (see, e.g.,~\cite[Proposition 2.35]{Rot06}) that every permutation of a finite $k$-element set $X$ factors as a (functional) composition of $k$ or fewer transpositions.
\item\label{it:F3} It is been known since J.\,A.~Erdos' seminal paper \cite{Er68} that every non-invertible matrix in the multiplicative monoid of the ring of $n$-by-$n$ matrices with entries in a field factors as a finite product of idempotent matrices, and later work has revealed that the same holds with fields replaced by a wider class of rings (see \cite[Sect.~1]{CoZaZa18} for a historical overview and recent developments).
\end{enumerate}
Roughly speaking, these results have all in common that they pertain to the \emph{existence} of a factorization of \emph{certain} elements of a monoid into a finite product of other elements that, in a sense, cannot be ``broken up into smaller pieces''. However, there is to date no general theory of factorization that gives shape and substance to this idea, and it is the primary goal of the present paper to start filling the gap.

The plan is as follows. First, we generalize the ordinary notions of unit, atom, and ir\-re\-duc\-i\-ble by pairing a monoid with a preorder (Definitions \ref{def:3.1}, \ref{def:3.4}, and \ref{def:3.6}). Next, assuming
a natural analogue of the ACCP (Definition \ref{def:3.8} and Remark \ref{rem:3.10}\ref{rem:3.10(3)}), we formulate an abstract factorization theorem (Theorem \ref{thm:3.11}) and discuss some of its applications, including the following:
\begin{enumerate}[label=(\textup{\textsc{a}\arabic{*}})]
\item\label{it:C1} A generalization to Dedekind-finite monoids (and, hence, to a variety of non-commutative rings with non-trivial zero divisors), of Anderson and Valdes-Leon's theorem on irreducible factorizations (Corollary \ref{cor:4.1}) and  Cohn's theorem on atomic factorizations (Corollary \ref{cor:4.6}).
\item\label{it:C2} An ``object decomposition theorem'' (Corollary \ref{cor:4.13}) for certain categories with finite products yielding as a special case a monoid-theoretic proof (Corollary \ref{cor:4.14}) that every $R$-module of finite uniform dimension over a possibly non-commutative ring $R$ (so in particular, every artinian or noetherian $R$-module) is a direct sum of fi\-nite\-ly many indecomposable $R$-modules (see Sect.~\ref{sec:4.3} for details and \cite{BaWi13} for a survey on monoid-theoretic methods applied to the study of modules).
\item\label{it:C3} An apparently new and, in a way, more conceptual proof of the folk theorem mentioned in item \ref{it:F2} above, where we characterize the transpositions of a finite set as a sort of ir\-re\-duc\-i\-ble elements associated with the fixed points of a permutation (Example \ref{exa:3.14}). 
\end{enumerate}
Among other things, \ref{it:C1} will allow us to recover Fan et al.'s extension of Cohn's theorem to ``nearly cancellative'' monoids, and to improve on a theorem of A.\,A.~Antoniou and the author \cite[Theorem 3.9]{An-Tr18} that characterizes atomicity in certain ``monoids of sets'' naturally arising from additive number theory and a\-rith\-me\-tic combinatorics: The latter is the content of Sect.~\ref{sec:4.2} (see, more specifically, Theorem \ref{thm:4.12}).

Further applications (especially to ``idempotent factorizations'' in matrix rings as outlined in item \ref{it:F3} above) are considered in a separate paper with L.~Cossu \cite{CoTr21} (see also the open questions in Sect.~\ref{sect:5}).
\section{Preliminaries.}
\label{sec:2}
In this section, we establish notation and terminology used all through the paper. Further notation and terminology, if not explained when first introduced, are standard or should be clear from context.

\subsection{Generalities}\label{sec:2.1}
We assume throughout that all relations are binary; all rings are non-zero, unital, and as\-so\-ci\-a\-tive; and all modules are left unital modules. We will usually be casual about the distinction between ``sets'' and ``classes'', but differentiating between these ``types'' will become relevant in Sect.~\ref{sec:4.3}, where among other things we need to guarantee that every category has a skeleton. With this in mind, we set out from the beginning to use Tarski-Grothendieck set theory as a foundation for the present work: Alternatives are possible (see, e.g., \cite[Sect.~1.7]{Fa19}), but the question is entirely beyond the scope of the paper.

We denote by $\mathbb N$ the (set of) non-negative integers, by $\bf Z$ the integers, and by $\mathbb R$ the real numbers. For all $a, b \in \allowbreak \mathbb R \cup \{\pm \infty\}$, we let $\llb a, b \rrb := \{x \in \mathbb Z\colon a \le x \le b\}$ be the \evid{discrete interval} between $a$ and $b$. Unless a statement to the contrary is made, we reserve the letters $\ell$, $m$, and $n$ (with or without subscripts or su\-per\-scripts) for positive integers; and the letters $i$, $j$, and $k$ for non-negative integers.

Given a set $X$ and an integer $k \ge 0$, 
we write $X^{\times k}$ for the Cartesian product of $k$ copies of $X$ and $|X|$ for the \evid{size} of $X$ (that is, $|X|$ is the number of elements of $X$ when $X$ is finite, and is $\infty$ otherwise). 

\subsection{Monoids.}
\label{sec:2.2}
We take a monoid to be a semigroup with an identity. Unless stated otherwise, monoids will typically be written multiplicatively and need not have any special property (e.g., commutativity). We refer the reader to \cite[Ch.~1]{Ho95} for basic aspects of semigroup theory.

Let $H$ be a monoid with identity $1_H$. An element $u \in H$ is \evid{right-invertible} (resp., \evid{left-invertible}) if $uv = \allowbreak 1_H$ (resp., $vu = 1_H$) for some $v \in H$. We use $H^\times$ for the set of \evid{units} (or \evid{invertible elements}) of $H$, that is, the elements of $H$ that are both left- and right-invertible: This means that $u \in H^\times$ if and only if there is a provably unique $v \in H$, called the \evid{inverse} of $u$ (in $H$) and denoted by $u^{-1}$, such that $uv = \allowbreak vu = 1_H$. It is well known (and elementary) that $H^\times$ is a subgroup of $H$, and we say that $H$ is 
\begin{itemize}[leftmargin=1cm]
\item \evid{reduced} if the only unit of $H$ is the identity, i.e., $H^\times = \{1_H\}$; 
	
\item \evid{cancellative} if $xz \ne yz$ and $zx \ne zy$ for all $x, y, z \in H$ with $x \ne y$;

\item \evid{Dedekind-finite} if every left- or right-invertible element is a unit, or equivalently, if $xy = 1_H$ for some $x, y \in H$ implies that at least one (and hence all) of $x$, $y$, and $yx$ is a unit.
\end{itemize}
The non-zero elements of a domain form a can\-cel\-la\-tive monoid under mul\-ti\-pli\-ca\-tion; and every can\-cel\-la\-tive or commutative monoid is Dedekind-finite (see also Proposition \ref{prop:4.3}\ref{prop:4.3(i)} and Remark \ref{rem:4.9}).

Given $X_1,\, \ldots,\, X_n \subseteq H$, we write $X_1 \cdots X_n$ for the the \evid{setwise product} of $X_1$ through $X_n$, that is, the set $\{x_1 \cdots x_n \colon x_1 \in X_1,\, \ldots,\, x_n \in X_n\} \subseteq H$; note that, if $X_i = \{x_i\}$ for some $i \in \llb 1, n \rrb$ and there is no likelihood of confusion, we will replace the set $X_i$ in the product $X_1 \cdots X_n$ with the element $x_i$. 

In particular, we denote by $X^n$ the setwise product of $n$ copies of a set $X \subseteq H$ and define $\mathrm{Mon}\langle X \rangle_H := X^0 \cup X^1 \cup X^2 \cup \cdots$, where $X^0 := \{1_H\}$. We call  $\mathrm{Mon}\langle X \rangle_H$ the \evid{submonoid} of $H$ \evid{generated by} $X$ and say that $H$ is a \evid{finitely generated} monoid if $H = \mathrm{Mon}\langle Y \rangle_H$ for a finite $Y \subseteq H$.

A \evid{monoid congruence} on $H$ is an equivalence relation $R$ on $H$ such that if $x \RR u$ and $y \RR v$ then $xy \RR uv$. If $R$ is a monoid congruence on $H$, we write $x \equiv y \bmod R$ in place of $x \RR y$ and say that ``$x$ is congruent to $y$ modulo $R\,$''. Consequently, we use $x \not\equiv y \bmod R$ to signify that $(x,y) \notin R$.

\subsection{Presentations}\label{sec:2.3}
In a couple of cases, we will consider monoids defined via generators and relations.
Therefore, we review some basic facts about presentations, cf.~\cite[Sect.~1.5]{Ho95}.

Let $X$ be a fixed set. We denote by $\mathscr F(X)$ the free monoid on $X$; use the symbols $\ast_X$ and $\varepsilon_X$, resp., for the operation and the identity of $\mathscr F(X)$; and refer to an element of $\mathscr F(X)$ as an \evid{$X$-word}, or simply as a \evid{word} if no confusion can arise. 
We recall that $\mathscr F(X)$ consists, as a set, of all finite tuples of elements of $X$; and $\mathfrak u \ast_X \mathfrak v$ is the \evid{con\-cat\-e\-na\-tion} of two such tuples $\mathfrak u$ and $\mathfrak v$. Accordingly, the identity of $\mathscr F(X)$ is the empty tuple (i.e., the unique element of $X^{\times 0}$), herein called the \evid{empty $X$-word}.

We take the \evid{length} of an $X$-word $\mathfrak u$, denoted by $\|\mathfrak u\|_X$, to be the unique non-negative integer $k$ such that $\mathfrak u \in X^{\times k}$; in particular, the empty word is the only $X$-word whose length is zero. Note that, if $\mathfrak u$ is an $X$-word of positive length $k$, then $\mathfrak u = u_1 \ast_X \cdots \ast_X u_k$ for some uniquely de\-ter\-mined $u_1,\, \ldots,\, u_k \in X$.

Given $z \in X$, we let the \evid{$z$-adic valuation} on $X$ be the function $\mathsf{v}_z^X \colon \mathscr F(X) \to \mathbb N$ that maps $\varepsilon_X$ to $0$ and a non-empty $X$-word $u_1 \ast_X \cdots \ast_X u_n$ of length $n$ to the number of indices $i \in \llb 1, n \rrb$ with $u_i = z$.

We will systematically drop the subscript (resp., superscript) $X$ from the above notation when there is no serious risk of ambiguity. As a result, we will write $\ast$ instead of $\ast_X$ and $\mathfrak u^{\ast k}$ for the $k^\mathrm{th}$ power of an $X$-word $\mathfrak u$, so that $
\mathfrak u^{\ast 0} := \varepsilon_X$ and $\mathfrak u^{\ast k} := \mathfrak u^{\ast(k-1)} \ast_X \mathfrak u$ for $k \in \mathbb N^+$.

With these premises in place, let $R$ be a relation on the free monoid $\mathscr F(X)$. We define $R^\sharp$ as the smallest monoid congruence on $\mathscr F(X)$ containing $R$, that is,
\[
R^\sharp := \bigcap\, \bigl\{\rho \subseteq \mathscr F(X) \times \mathscr F(X)\colon \rho \text{ is a monoid congruence on } \mathscr F(X) \text{ with } R \subseteq \rho \bigr\}.
\]
This means that $\mathfrak u \equiv \mathfrak v \bmod R^\sharp$ if and only if there are $\mathfrak z_0,\, \mathfrak z_1, \, \ldots,\, \mathfrak z_n \in \mathscr F(X)$ with $\mathfrak z_0 = \mathfrak u$ and $\mathfrak z_n = \mathfrak v$ such that, for each $i \in \llb 0, n-1 \rrb$, there exist $X$-words $\mathfrak p_i$, $\mathfrak q_i$, $\mathfrak q_i^\prime$, and $\mathfrak r_i$ with the following properties:
\begin{center}
(i) either $\mathfrak q_i = \mathfrak q_i^\prime$, or $\mathfrak q_i \RR \mathfrak q_i^\prime$, or $\mathfrak q_i^\prime \RR \mathfrak q_i$; \hskip 1cm (ii) $\mathfrak z_i = \mathfrak p_i \ast \mathfrak q_i \ast \mathfrak r_i$ and $\mathfrak z_{i+1} = \mathfrak p_i \ast \mathfrak q_i^\prime \ast \mathfrak r_i$.
\end{center}
We denote by $\mathrm{Mon}\langle X \mid R \rangle$ the monoid  obtained by taking the quotient of $\mathscr F(X)$ by the congruence $R^\sharp$.
We write $\mathrm{Mon}\langle X \mid R \rangle$ multiplicatively and call it a (\evid{monoid}) \evid{presentation}; in particular, $\mathrm{Mon}\langle X \mid R \rangle$ is a \evid{finite presentation} if $X$ and $R$ are both finite sets. We refer to the elements of $X$ as the \evid{generators} of the presentation, and to each pair $(\mathfrak q, \mathfrak q^\prime) \in R$ as a \evid{defining relation}. 
If there is no danger of confusion, we identify, as is customary, an $X$-word $\mathfrak z$ with its equivalence class in $\mathrm{Mon}\langle X \mid R \rangle$.

The \evid{left graph} of a presentation $\mathrm{Mon}\langle X \mid R\rangle$ is the undirected multigraph with vertex set $X$ and an edge from $y$ to $z$ for each pair $(y \ast \mathfrak y, z \ast \mathfrak z) \in R$ with $y, z \in X$ and
$\mathfrak y, \mathfrak z \in \mathscr F(X)$; note that this results in a loop when $y = z$, and in multiple (or parallel) edges between $y$ and $z$ if there are two or more defining relations of the form $(y \ast \mathfrak y, z \ast \mathfrak z)$.
The \evid{right graph} of a presentation is defined analogously, using the
right-most (instead of left-most) letters of the words in a defining relation.

A monoid is \evid{Adian} if it is isomorphic to a finite presentation 
whose left and right graphs are \emph{cycle-free}, that is, contain no cycles (including loops). Our interest for Adian monoids stems from the following:

\begin{theorem}
	\label{th:adian-theorem}
	Every Adian monoid embeds into a group
	\textup{(}and hence is cancellative\textup{)}.
\end{theorem}
The result is attributed to S.\,I.~Adian \cite[Theorem II.4]{Ad66}, and therefore it is commonly referred to as \emph{Adian's embedding theorem}; it will come in useful in Example \ref{exa:4.8}.

\section{Preorders and their Interplay with Monoids}
\label{sec:3}

In the present section, we aim to generalize fundamental aspects of the classical theory of factorization by combining the language of monoids with that of preorders: This will prepare the ground for the abstract factorization theorem (Theorem \ref{thm:3.11}) promised in the introduction. The section also includes a variety of examples that will help illustrate some key points: Certain of these examples are of independent interest and we will return to them later, when discussing applications in Sect.~\ref{sec:4}. 

We start with the following definition (see, e.g., \cite[Definition 3.1]{Ha06} and note that some authors prefer the terms ``pre-order'', ``quasi-order'', or ``quasi-ordering'' to the term ``preorder''):

\begin{definition}\label{def:3.1}
Let $X$ be a set. A \evid{preorder} on $X$ is a relation $R$ on $X$ such that $x \RR x$ for all $x \in X$ (i.e., $R$ is \emph{reflexive}), and $x \RR z$ whenever $x \RR y$ and $y \RR z$ (i.e., $R$ is \emph{transitive}). 

In particular, we say a preorder $R$ on $X$ is \evid{total} if, for all $x, y \in X$, $x \RR y$ or $y \RR x$; and is an \evid{order} if $x \RR y$ and $y \RR x$ imply $x = y$ \textup{(}i.e., $R$ is \emph{antisymmetric}\textup{)}.
\end{definition}

We will usually denote a preorder on a set $X$ by either of the relational symbols $\leq$ and $\preceq\,$, with or without subscripts or superscripts. In particular, we reserve the symbol $\le$, without subscripts or super\-scripts, for the standard order on $\mathbb R \cup \{\pm \infty\}$ and its subsets.

\begin{definition}\label{def:3.2}
Given a preorder $\preceq$ on a set $X$ and elements $x, y \in X$, we say $x$ is \textsf{$\preceq$-equivalent} to $y$ if $x \preceq y \preceq x$, and we write $x \prec y$ to signify that $x \preceq y$ and $y \npreceq x$. Accordingly, we say a sequence $(x_k)_{k \ge 0}$ of elements of $X$ is  \evid{$\preceq$-non-increasing} \textup{(}resp., \evid{$\preceq$-decreasing}\textup{)} if $x_{k+1} \preceq x_k$ \textup{(}resp., $x_{k+1} \prec x_k$\textup{)} for every $k \in \mathbb N$; and is \evid{$\preceq$-non-decreasing} \textup{(}resp., \evid{$\preceq$-increasing}\textup{)} if $x_k \preceq x_{k+1}$ \textup{(}resp., $x_k \prec x_{k+1}$\textup{)} for every $k \in \mathbb N$.
\end{definition}

It is perhaps worth stressing that, for a preorder $\preceq\,$, the condition ``$x \prec y$'' is stronger than ``$x \preceq y$ and $x \ne y\,$''; the two conditions are equivalent if and only if $\preceq$ is an order. Also, note that ``$\preceq$-decreasing'' means ``strictly $\preceq$-decreasing''; similarly, ``$\preceq$-increasing'' means ``strictly $\preceq$-increasing''.

\begin{examples}\label{exa:3.3}
\begin{enumerate*}[resume,label=\textup{(\arabic{*})}]
\item\label{exa:3.3(1)} Let $\preceq$ be a preorder (resp., an order) on a set $X$. The relation $\preceq^\mathrm{op}$ on $X$ defined by taking $x \preceq^\mathrm{op} y$ if and only if $y \preceq x$, is still a preorder (resp., an order) on $X$: We will refer to $\preceq^\mathrm{op}$ as the \evid{dual preorder} (resp., the \evid{dual order}) of $\preceq$, or simply as the \evid{dual} of $\preceq$. It is common to denote the preorder $\preceq^\mathrm{op}$ by the ``dual'' of the relational symbol $\preceq$ (that is, by $\succeq$). However, we will not do so, except for the dual of the standard order $\leq$ on $\mathbb R \cup \{\pm \infty\}$ and its subsets, which, as usual, we denote by $\geq$.
\end{enumerate*}

\vskip 0.05cm

\begin{enumerate*}[resume,label=\textup{(\arabic{*})}]
\item\label{exa:3.3(4)} Given a function $\phi\colon X \to Y$ and a preorder $\preceq$ on $Y$, the relation $\preceq_\phi$ on $X$ defined for all $x, y \in X$ by $x \preceq_\phi y$ if and only if $\phi(x) \preceq \phi(y)$, is a preorder on $X$. We will refer to $\preceq_\phi$ as the \evid{pullback preorder} induced by $\preceq$ through $\phi$ or, more simply, as the $\phi$-pullback of $\preceq$.
\end{enumerate*}
\end{examples}

Here, we are mainly interested in the interaction between preorders and monoids. The basic idea is nothing new (see, e.g., \cite[Sect.~1.2]{Fa19}) and leads straight to the following:

\begin{definition}\label{def:3.4}
We let a \evid{premonoid} be a pair $(H, \preceq$) consisting of a monoid $H$ and a preorder on (the underlying set of) $H$. In particular, we say that the premonoid $(H, \preceq)$ is a \evid{preordered monoid} if $xy \preceq uv$ whenever $x \preceq u$ and $y \preceq v$; 
and a \evid{linearly preordered} monoid if it is a preordered monoid with the further property that $\preceq$ is a total preorder and $x \prec y$ implies $uxv \prec uyv$ for all $u, v \in H$. \evid{Ordered monoids} and \evid{linearly ordered monoids} are defined in a similar fashion, by requiring that the preorder $\preceq$ is an order.
\end{definition}

There are in principle many preorders one can put on a monoid $H$: Most notably, those considered in the next example are of utmost importance in the classical theory of factorization (we will pay special attention to them in Sect.~\ref{sec:4.1}); their duals (in the terminology of Example \ref{exa:3.3}\ref{exa:3.3(1)}) were thoroughly studied by J.\,A.~Green in \cite{Gr51} and hence are often called the \emph{Green preorders}.

\begin{examples}\label{exa:3.5}
\begin{enumerate*}[label=\textup{(\arabic{*})}]
\item\label{exa:3.5(4)} It is a simple exercise to show that the relation $\mid_H$ on $H$ defined for all $x, y \in H$ by $x \mid_H y$ (read ``$x$ divides $y$'') if and only if $y \in HxH$, is a preorder on $H$. We will refer to $\mid_H$ as the \evid{divisibility preorder} on $H$ and write $x \nmid_H y$ (read ``$x$ does not divide $y$'') if $y \notin HxH$. In general, $\mid_H$ is not an order, as seen, e.g., by considering the case where $H^\times$ is a non-trivial group (and noting that every unit of $H$ divides any other). Moreover, $(H, \mid_H)$ need not be a preordered monoid: In the free monoid $\mathscr F(A)$ on the $2$-element set $A = \{a, b\}$, we have that $a \mid_{\mathscr F(A)} a$ and $a \mid_{\mathscr F(A)} b \ast a$, but $a \ast a \nmid_{\mathscr F(A)} a \ast b \ast a$.
\end{enumerate*}

\vskip 0.05cm

\begin{enumerate*}[resume,label=\textup{(\arabic{*})}]
\item\label{exa:3.5(5)} Let  $\vdash_H$ and $\dashv_H$ be, resp., the relations on $H$ defined for all $x, y \in H$ by $x \vdash_H y$ (read ``$x$ divides $y$ from the left'') if and only if $y \in xH$, and $x \dashv_H y$ (read ``$x$ divides $y$ from the right'') if and only if $y \in Hx$. 
It is clear that each of $\vdash_H$ and $\dashv_H$
is a preorder on $H$; and it is a simple exercise to check that the divisibility preorder $\mid_H$ on $H$ is the \textsf{transitive closure} of the relation $\vdash_H \cup \dashv_H$ (recall that a relation on $H$ is nothing else that a subset of $H \times H$), meaning that $x \mid_H y$ if and only if there exist $z_0, \ldots, z_n \in H$ with $z_0 = x$ and $z_n = y$ such that, for every $k \in \llb 0, n-1 \rrb$, either $z_k \vdash_H z_{k+1}$ or $z_k \dashv_H z_{k+1}$.
\end{enumerate*}
\end{examples}

One of the key insights of this whole work is that every premonoid comes with a natural generalization of the notion of unit, which, in turn, results in a natural generalization of the notions of atom and irreducible discussed in the introduction (see Remark \ref{rem:3.7}\ref{rem:3.7(2)}). More precisely, we have the following:

\begin{definition}
\label{def:3.6}
Let $(H, \preceq)$ be a premonoid. An element $u \in H$ is a \evid{$\preceq$-unit} \textup{(}of $H$\textup{)} if $u$ is $\preceq$-equivalent to $1_H$ \textup{(}i.e., $u \preceq 1_H \preceq u$\textup{)}; otherwise, $u$ is a \evid{$\preceq$-non-unit}.
Accordingly, a $\preceq$-non-unit $a \in H$ is called
\begin{itemize}[leftmargin=1cm]
\item a \evid{$\preceq$-irreducible} (of $H$) if $a \ne xy$ for all $\preceq$-non-units $x, y \in H$ with $x \prec a$ and $y \prec a$;
\item a \evid{$\preceq$-atom} if $a \ne xy$ for all $\preceq$-non-units $x, y \in H$;
\item a \evid{$\preceq$-quark} if there exists no $\preceq$-non-unit $b \in H$ with $b \prec a$;
\item a \evid{$\preceq$-prime} if $a \preceq xy$, for some $x, y \in H$, implies $a \preceq x$ or $a \preceq y$.
\end{itemize}
We say that $H$ is \evid{$\preceq$-factorable} if each $\preceq$-non-unit factors as a \textup{(}non-empty, finite\textup{)} product of $\preceq$-ir\-re\-duc\-i\-bles; and \evid{$\preceq$-atomic} if each $\preceq$-non-unit factors as a product of $\preceq$-atoms. 
\end{definition}

It is actually the notion of $\preceq$-irreducible as per the above definition that is central to the study of factorization from the perspective of this work: The other notions are somewhat secondary in importance, though still of interest due, e.g., to the fact that understanding the interrelation between $\preceq$-irreducibles, $\preceq$-atoms, $\preceq$-quarks, and $\preceq$-prime in a specific scenario is often pivotal to a deeper comprehension of various phenomena (see, for instance, Propositions \ref{prop:3.13} and \ref{prop:4.3} and Theorem \ref{thm:4.12}).

\begin{remarks}\label{rem:3.7}
\begin{enumerate*}[label=\textup{(\arabic{*})}]
\item\label{rem:3.7(1)}
The rationale behind Definition \ref{def:3.6} is vaguely reminiscent of certain ideas set forth in \cite{Ba-Sm15}, where, among other things, N.\,R.~Baeth and D.~Smertnig axiomatize a notion of ``divisibility relation'' (ibid., Definition 5.1): Every divisibility relation corresponds to a notion of ``prime-like element'' (ibid., Definition 5.3), similarly to how a preorder $\preceq$ on a monoid $H$ comes by with a corresponding notion of $\preceq$-irreducible, $\preceq$-atom, and so on. But while Baeth and Smertnig's approach is firmly anchored to a classical paradigm of factorization (as seen, e.g., from the critical role that ``ordinary units'' keep playing in their frame\-work), this is not the case with our approach. Moreover, it is arguable that Baeth and Smertnig's notion of prime-like element is not really a generalization of the classical notion of atom, but rather a generalization of the euclidean notion of prime number, which is in turn generalized by the notion of $\preceq$-prime (we will not discuss $\preceq$-primes any further in this work).
\end{enumerate*}

\vskip 0.05cm

\begin{enumerate*}[resume,label=\textup{(\arabic{*})}]
\item\label{rem:3.7(2)} 
Let $H$ be a monoid. In the notation of Example \ref{exa:3.5}, $a \in H$ is a $\mid_H$-irreducible if and only if $a$ is a $\mid_H$-non-unit and $a \ne xy$ for all $\mid_H$-non-units $x, y \in H$ such that $a \nmid_H x$ and $a \nmid_H y$. In general, this is the best (if trivial) characterization of $\mid_H$-irreducibility we can hope for (it goes the same with $\mid_H$-atoms and $\mid_H$-quarks), mainly because there is no sensible way to characterize the $\mid_H$-units without imposing restrictions on $H$. But assume from now on that $H$ is a Dedekind-finite monoid.
\\

\indent{}It is then easily checked that $u \in H$ is a $\vdash_H$-unit (i.e., $1_H \in uH$) if and only if $u$ is a $\dashv_H$-unit (i.e., $1_H \in Hu$), if and only if $u$ is a $\mid_H$-unit (i.e., $1_H \in HuH$), if and only if $u$ is a unit. Therefore, $a \in H$ is a $\vdash_H$-atom (i.e., a $\vdash_H$-non-unit with $a \ne xy$ for all $\vdash_H$-non-units $x,y \in H$) if and only if $a$ is a $\dashv_H$-atom (i.e., a $\dashv_H$-non-unit with $a \ne xy$ for all $\dashv_H$-non-units $x,y \in H$), if and only if $a$ is a $\mid_H$-atom (i.e., a $\mid_H$-non-unit with $a \ne xy$ for all $\mid_H$-non-units $x,y \in H$), if and only if $a$ is an atom (i.e., $a$ is a non-unit and $a \ne xy$ for all non-units $x, y \in H$). Similarly, $a \in H$ is a $\mid_H$-irreducible if and only if $a$ is irreducible (i.e., $a$ is a non-unit such that $a \ne xy$ for all non-units $x, y \in H$ with $HxH \ne HaH \ne HyH$). 
\end{enumerate*}

\vskip 0.05cm

\begin{enumerate*}[resume,label=\textup{(\arabic{*})}]
\item\label{rem:3.7(3)} If the set of atoms of a monoid $H$ is non-empty, then $H$ is Dedekind-finite (see \cite[Lemma 2.2(i)]{Fa-Tr18} for details). It thus follows from item \ref{rem:3.7(2)} that either $H$ has no atoms; or every atom is a $\mid_H$-atom, and vice versa. In particular, $H$ is \evid{atomic} (i.e., every non-unit factors as a product of atoms) if and only if $H$ is $\mid_H$-atomic (as per Definition \ref{def:3.6}) and Dedekind-finite.
\end{enumerate*}

\vskip 0.05cm

\begin{enumerate*}[resume,label=\textup{(\arabic{*})}]
\item\label{rem:3.7(4)} Let $(H, \preceq)$ be a premonoid. It is straightforward from Definition \ref{def:3.6} that, if $a \in H$ is a $\preceq$-atom or a $\preceq$-quark, then $a$ is also a $\preceq$-irreducible. In general, the converse is not true.
\\

\indent{}E.g., let $H$ be the multiplicative monoid of a (commutative or non-commutative) domain $R$. The zero $0_R$ of $R$ is a $\mid_H$-irreducible of $H$, because $0_R \ne xy$ for all $x, y \in R \setminus \{0_R\}$. However, $0_R$ is not a $\mid_H$-atom of $H$, for $0_R$ is not a $\mid_H$-unit and $0_R = 0_R\, 0_R$ (by item \ref{rem:3.7(2)}, the $\mid_H$-units of $H$ are precisely the units, since $H$ is Dedekind-finite). If, in addition, $R$ is not a skew field, then $0_R$ is not a $\mid_H$-quark either: Just let $x$ be a non-zero non-unit of $R$ and consider that $x \mid_H 0_R$ but $0_R \nmid_H x$.
\end{enumerate*}
\end{remarks}

It is a natural question to look for conditions under which the elements of a certain subset $S$ of a monoid $H$ \evid{factor through} the elements of another set $A \subseteq H$, in the sense that $S \subseteq \textup{Mon} \langle A \rangle_H$: Here, we aim to provide a partial answer to this question in the case where, given a preorder $\preceq$ on $H$, we take $S$ to be the set of $\preceq$-non-units (of $H$) and $A$ be either the set of $\preceq$-irreducibles, the set of $\preceq$-atoms, or the set of $\preceq$-quarks. 
Most notably, we aim to obtain \emph{sufficient} conditions for $H$ to be $\preceq$-factorable that extend the ideal-theoretic conditions reviewed in the introduction (see Remark \ref{rem:3.10}\ref{rem:3.10(3)} for additional details).

\begin{definition}\label{def:3.8}
We say that a preorder $\preceq$ on a set $X$
\begin{itemize}
\item is \evid{artinian} or satisfies the \evid{descending chain con\-di\-tion} \textup{(DCC)} if, for every $\preceq$-non-increasing sequence $(x_k)_{k \ge 0}$ of elements of $X$, there exists $k' \in \mathbb N$ such that, for $k \ge k'$, $x_k \preceq x_{k+1}$;
\item is \evid{noetherian} 
or satisfies the \evid{ascending chain condition} \textup{(ACC)} 
if the dual $\preceq^\mathrm{op}$ of $\preceq$ is artinian.
\end{itemize}
We call a premonoid $(H, \preceq)$ \evid{artinian} (resp., \evid{noetherian}) if the preorder $\preceq$ is artinian (resp., noetherian).
\end{definition}

In other terms, a preorder $\preceq$ on a set $X$ is artinian (resp., noetherian) if and only if there is no sequence $(x_k)_{k \ge 0}$ of elements of $X$ with $x_{k+1} \prec x_k$ (resp., $x_k \prec x_{k+1}$) for all $k \in \mathbb N$. See, e.g., \cite[Ch.~8, Definition 1.1]{Ha06}, where the term ``well-founded'' is however preferred to the term ``artinian''; or \cite[Definition 2.2]{Gal91}, where the term ``noetherian'' is used in a way that is dual to how we use it here.

In the remainder, we will employ the word ``artinianity'' (resp., ``noetherianity'') to refer to the prop\-er\-ty that a preorder or a premonoid is artinian (resp., noetherian).

\begin{remark}\label{rem:3.10}
\begin{enumerate*}[label=\textup{(\arabic{*})}]
\item\label{rem:3.10(1)}
Let $\preceq$ be a preorder on a set $X$ and assume there is a function $\lambda$ from $X$ to (a subset of) $ \mathbb N \cup \{\infty\}$ such that $\lambda(x) < \lambda(y)$ whenever $x \prec y$. Then $\preceq$ is artinian, or else there would exist a se\-quence $(N_k)_{k \ge 0}$ of non-negative integers with $N_{k+1} < N_k$ for each $k \in \mathbb N$ (absurd). In particular, note that, if $H$ is \emph{acyclic} as per Definition \ref{def:4.2} and $\preceq$ is the divisibility preorder $\mid_H$, then $\lambda$ is a \emph{length function} in the sense of \cite[Definition 1.1.3.2]{GeHK06} (the commutative case) and \cite[Definition 2.26]{Fa-Tr18}.
\end{enumerate*}

\vskip 0.05cm

\begin{enumerate*}[resume,label=\textup{(\arabic{*})}]
\item\label{rem:3.10(1b)}
Every preorder $\preceq$ on a finite set $X$ is artinian. In fact, let $\lambda$ be the function $X \to \mathbb N$ that maps an element $x \in X$ to the largest $k \in \mathbb N$ for which there are $x_0,\, \ldots,\, x_k \in X$ with $x_0 = x$ and $x_{i+1} \prec x_i$ for each $i \in \llb 0, k-1 \rrb$. Since $\prec$ is a transitive relation on $X$ and $x \prec y$ implies $x \ne y$, the fi\-nite\-ness of $X$ guarantees the well-definiteness of $\lambda$ (by the Pigeonhole Principle). It is then clear by construction that $x \prec y$ yields $\lambda(x) < \lambda(y)$. So, by item \ref{rem:3.10(1)}, $\preceq$ is artinian (as wished).
\end{enumerate*}

\vskip 0.05cm

\begin{enumerate*}[resume,label=\textup{(\arabic{*})}]
\item\label{rem:3.10(2)} Assume that $\preceq$ is an artinian preorder on a set $X$, and let $S$ be a non-empty subset of $X$. Then it is well known that $S$ has at least one \evid{$\preceq$-minimal element}, meaning that there exists $\bar{x} \in S$ such that, if $y \preceq \bar{x}$ for some $y \in S$, then $\bar{x} \preceq y$: We include the short proof here for the sake of completeness (the argument relies on the Axiom of Choice, which is actually a theorem \cite[Sect.~2]{Tar38} in Tarski-Grothendieck set theory --- i.e., in the axiomatic system we are using as a foundation, as explained in Sect.~\ref{sec:2.1}).
\\[0.1em]\indent{}

To begin, choose an arbitrary $x \in S$ (this is possible because $S$ is not the empty set). Next, recursively define an $S$-valued sequence $(x_k)_{k \ge 0}$ as follows: Start with $x_0 := x$. If, for some $k \in \mathbb N$, $x_k$ is not a $\preceq$-min\-i\-mal element of $S$, then pick $y \in S$ such that $y \prec x_k$ and set $x_{k+1} := y$; otherwise, set $x_{k+1} := x_k$. Since $\preceq$ is assumed to be artinian, there exists $k_0 \in \mathbb N$ such that $x_{k+1}$ is $\preceq$-equivalent to $x_k$ for every $k \ge k_0$; and by construction of the sequence $(x_k)_{k \ge 0}$, this means that $x_{k_0}$ is a $\preceq$-minimal element of $S$.
\end{enumerate*}

\vskip 0.05cm

\begin{enumerate*}[resume,label=\textup{(\arabic{*})}]
\item\label{rem:3.10(3)} 
In the notation of items \ref{exa:3.5(4)} and \ref{exa:3.5(5)} of Example \ref{exa:3.5}, $H$ satisfies the ACCP (as formulated in the first paragraph of the introduction) 
if and only if the divisibility preorder $\mid_H$ is  artinian, while $H$ satisfies the ACCPR (resp., ACCPL) if and only if the ``divides from the left'' preorder $\vdash_H$ (resp., the ``divides from the right'' preorder $\dashv_H$) is artinian. 
\end{enumerate*}
\end{remark}

At long last, we are ready for the main 
result of the paper: The statement is quite general and, because of its very generality, rather easy to prove.
\begin{theorem}\label{thm:3.11}
Let $\preceq$ be an artinian preorder on a monoid $H$. Then every $\preceq$-non-unit of $H$ factors as a product of $\preceq$-irreducibles \textup{(}namely, $H$ is $\preceq$-factorable\textup{)}.
\end{theorem}

\begin{proof}
Let $\Omega$ be the set of $\preceq$-non-units of $H$ that do not factor as a product of $\preceq$-irreducibles, and suppose for a contradiction that $\Omega$ is non-empty. By Remark \ref{rem:3.10}\ref{rem:3.10(2)}, $\Omega$ has a $\preceq$-minimal element $\bar{x}$. In particular, $\bar{x}$ is a $\preceq$-non-unit, but not a $\preceq$-irreducible. So, $\bar{x} = yz$ for some $\preceq$-non-units $y, z \in H$ with $y \prec \bar{x}$ and $z \prec \bar{x}$. 
But this is only possible if $y \notin \Omega$ and $z \notin \Omega$, since $\bar{x}$ is a $\preceq$-minimal element of $\Omega$. Therefore,
each of $y$ and $z$ factors as a product of $\preceq$-irreducibles; whence the same is also true for $\bar{x}$ (absurd). 
\end{proof}

Theorem \ref{thm:3.11} applies to a wide range of different situations. But before turning to applications, we aim to show that, in addition to the mere existence of certain factorizations, one can say a little more about the ``arithmetic of a premonoid\,'' $(H, \preceq)$ when $H$ and $\preceq$ are related by a condition that, while much stronger than artinianity, is often met in practice.

\begin{definition}\label{def:3.12}
Given a premonoid $(H, \preceq)$ and an element $x \in H$, we denote by $\hgt_\preceq^H(x)$ the supremum of the set of all $n \in \mathbb N^+$ for which there exist $\preceq$-non-units $x_1, \,\ldots,\, x_n \in H$ with $x_1 = x$ and $x_{i+1} \prec x_i$ for each $i \in \llb 1, n-1 \rrb$, where $\sup \emptyset := 0$. We call $\hgt_\preceq^H(x)$ the \evid{$\preceq$-height} of $x$ (relative to the monoid $H$) and say $(H, \preceq)$ is a \evid{strongly artinian} premonoid if $\hgt_\preceq^H(y) < \infty$ for every $y \in H$. We will usually write $\hgt(\cdot)$ in place of $\hgt_\preceq^H(\cdot)$ if no confusion can arise.
\end{definition}

Definition \ref{def:3.12} is resonant with the notions of ``ideal height'' and ``Krull dimension'' in ring theory: This is no coincidence and we hope to discuss the details in future work.

\begin{proposition}\label{prop:3.13}
Let $(H, \preceq)$ be a strongly artinian premonoid and suppose that, for each $x \in H$ that is neither a $\preceq$-unit nor a $\preceq$-quark, there are $\preceq$-non-units $y, z \in H$ with $y \preceq x$ and $z \preceq x$ such that $x = yz$ and $\hgt(y) + \hgt(z) \le \hgt(x)$. The following hold:
	
\begin{enumerate}[label=\textup{(\roman{*})}]
\item\label{prop:3.13(i)} The preorder $\preceq$ is artinian, the monoid $H$ is $\preceq$-factorable, and every $\preceq$-irreducible is a $\preceq$-quark.
\item\label{prop:3.13(ii)} Every $\preceq$-non-unit $x \in H$ factors into a non-empty product of $\hgt(x)$ or fewer $\preceq$-quarks.
\end{enumerate}
\end{proposition}

\begin{proof}
\ref{prop:3.13(i)} As the $\preceq$-height of each element of $H$ is finite (by hypothesis), the function $\lambda \colon H \to \mathbb N \colon x \mapsto \hgt(x)$ is well defined. In particular, it is obvious from Definition \ref{def:3.12} that
$\lambda(u) = \hgt(u) < \hgt(v) = \lambda(v)$ for all $u, v \in H$ with $u \prec v$.
Therefore, we get from Remark \ref{rem:3.10}\ref{rem:3.10(1)} that $\preceq$ is an artinian preorder; and by Theorem \ref{thm:3.11}, this yields that $H$ is a $\preceq$-factorable monoid. 

We are left to check that every $\preceq$-irreducible of $H$ is also a $\preceq$-quark. Let $x \in H$ be neither a $\preceq$-unit nor a $\preceq$-quark. It is then guaranteed by our hypotheses that there exist $\preceq$-non-units $y, z \in H$ with $y \preceq x$ and $z \preceq x$ such that $x = yz$ and $\hgt(y) + \hgt(z) \le \hgt(x)$. It follows $y \prec x\,$; otherwise, $y$ is $\preceq$-equivalent to $x$ and, hence, $\hgt(x) = \hgt(y)$ and $\hgt(z) = 0$, which is impossible because the $\preceq$-units are the only elements in $H$ whose $\preceq$-height is zero. Likewise, we see that $z \prec x$. In consequence, $x$ is not a $\preceq$-irreducible.

\vskip 0.05cm
	
\ref{prop:3.13(ii)} Let $x$ be a $\preceq$-non-unit of $H$ and set $n := \hgt(x)$. If $n = 1$, then $x$ is a $\preceq$-quark and the conclusion is trivial. Hence assume $n \ge 2$, and suppose inductively that every $\preceq$-non-unit of $H$ of $\preceq$-height $h \le n-1$ factors as a product of $h$ or fewer $\preceq$-quarks. Since $x$ is neither a $\preceq$-unit nor a $\preceq$-quark, we have, similarly as in the proof of part \ref{prop:3.13(i)}, that $x = yz$ for some $\preceq$-non-units $y, z \in H$ with $\hgt(y) + \hgt(z) \le n$, which in turn yields $1 \le \hgt(y) < n$ and $1 \le \hgt(z) < n$. So, by the inductive hypothesis, there exist $k \in \llb 1, \hgt(y) \rrb$ and $\ell \in \llb 1, \hgt(z) \rrb$ such that $y = a_1 \cdots a_k$ and $z = b_1 \cdots b_\ell$ for certain $\preceq$-quarks $a_1, \,\ldots,\, a_k, \, b_1,\, \ldots, \, b_\ell \in H$. Then $x = yz = a_1 \cdots a_k b_1 \cdots b_\ell$; and this is enough to finish the proof (by induction on $n$), upon con\-sid\-er\-ing that $k+\ell \le \hgt(x)$.
\end{proof}

As a first bench test for the ideas heretofore set forth, we are going to apply Proposition \ref{prop:3.13} to a classical problem in group theory (further applications will be discussed in Sect.~\ref{sec:4}).

\begin{example}\label{exa:3.14}
Let $X$ be a finite $k$-element set and $\mathfrak S(X)$ the \evid{symmetric group} of $X$, that is, the set of all permutations of $X$ endowed with the operation of (\evid{functional}) \evid{composition} $\circ \colon \mathfrak S(X) \times \mathfrak S(X) \to \mathfrak S(X)$ that maps a pair $(f, g)$ of permutations of $X$ to the permutation $f \circ g \colon X \to X \colon x \mapsto f(g(x))$. We will denote the identity of $\mathfrak S(X)$ (viz., the identity function on $X$) by $\mathrm{id}_X$.

It is well known (see the introduction) that every $f \in \mathfrak S(X)$ factors as a composition of \evid{transpositions}, i.e., permutations of $X$ that exchange two elements and keep all others fixed. We aim to give a new proof of this result, by showing that, for $k \ge 2$, every $f \in \mathfrak S(X) \setminus \{\mathrm{id}_X\}$ is a composition of $k-1-|\mathrm{Fix}(f)|$ or fewer transpositions, where $\mathrm{Fix}(f) := \{x \in X \colon f(x) = x\}$ is the set of fixed points of $f$: The bound $k - 1 - |\mathrm{Fix}(f)|$ is sharp but not best possible (see, e.g., \cite{Ma95} and references therein); that, however, is not the point here. (For $k = 0$ or $k = 1$, $\mathfrak S(X)$ is a one-element group and there is nothing to prove.)

To begin, assume $k \ge 2$ and let $\preceq$ be the dual of the pullback of the standard order $\leq$ on $\mathbb N$ through the function $\phi \colon \mathfrak S(X) \to \mathbb N \colon f \mapsto |\mathrm{Fix}(f)|$
(see items \ref{exa:3.3(1)} and  \ref{exa:3.3(4)} of Example \ref{exa:3.3} for the terminology); to wit, we have that $f \preceq g$, for some $f, g \in \mathfrak S(X)$, if and only if $|\mathrm{Fix}(g)| \le |\mathrm{Fix}(f)|$. 
Clearly, a permutation $f$ of $X$ has $k-1$ or more fixed points if and only if $f = \mathrm{id}_X$. It follows that $(\mathfrak S(X), \preceq)$ is a strongly artinian pre\-monoid with $
\hgt(f) = k - 1 - |\mathrm{Fix}(f)| \le k-1$ for each $f \in \mathfrak S(X) \setminus \{\mathrm{id}_X\}$; 
whence the $\preceq$-quarks of $\mathfrak S(X)$ are the permutations with exactly $k-2$ fixed points, that is, the transpositions.

Now, suppose that $f \in \mathfrak S(X)$ is neither the identity $\mathrm{id}_X$ nor a transposition, so that $\phi(f) = |\mathrm{Fix}(f)| \le k-3$. Accordingly, pick an element $\bar{x} \in X$ that is not a fixed point of $f$, and let $\tau$ be the transposition that exchanges $\bar{x}$ and $f(\bar{x})$. On the one hand, $
f = \tau \circ (\tau \circ f) = (\tau \circ \tau) \circ f = \mathrm{id}_X \circ f = f$.
On the other, it is readily checked that $\mathrm{Fix}(f) \cup \{\bar{x}\} \subseteq \mathrm{Fix}(\tau \circ f)$; note, in particular, that $f(\bar{x})$ is not a fixed point of $f$, or else we would have $f(f(\bar{x})) = f(\bar{x})$, contradicting that $f$ is injective. Consequently, $\tau \circ f$ has more fixed points than $f$ and, hence, $\hgt(\tau) + \hgt(\tau \circ f) = 1 + \hgt(\tau \circ f) \le \hgt(f)$.

So, putting it all together, we can conclude from Proposition \ref{prop:3.13} that every $f \in \mathfrak S(X) \setminus \{\mathrm{id}_X\}$ factors as a composition of $k-1 - |\mathrm{Fix}(f)|$ or fewer transpositions (as wished).
\end{example}

We close the section by remarking that the artinianity of a premonoid $(H, \preceq)$, while sufficient for each $\preceq$-non-unit of $H$ to factor as a product of $\preceq$-irreducibles (Theorem \ref{thm:3.11}), is not necessary: Just let $H$ be the multiplicative monoid of the non-zero elements of the integral domain constructed by A.~Grams in \cite[Sect.~1]{Gr74} and $\preceq$ be the divisibility preorder $\mid_H$ on $H$; and consider that the $\mid_H$-irreducibles of a cancellative, commutative monoid $H$ are precisely the atoms of $H$ (Corollary \ref{cor:4.4}). A different construction will be pre\-sent\-ed in Example \ref{exa:4.8}, where, among other things, we show that it is even possible for a monoid $H$ to be reduc\-ed, finitely generated, cancellative, and $\mid_H$-atomic, and yet not satisfy the ACCP (note that, by \cite[Proposition 2.7.4.2]{GeHK06}, this can only happen if $H$ is not commutative).

\section{Applications}\label{sec:4}

Below we discuss some applications of Theorem \ref{thm:3.11}. We organize the section into three subsections:  Sect.~\ref{sec:4.1} is all about the classical theory of factorization in a ``nearly cancellative'' setting, but the approach is totally non-classical as we embrace the point of view of Sect.~\ref{sec:3}. In Sect.~\ref{sec:4.2}, we focus attention on \emph{power monoids}, a ``highly non-cancellative'' class of monoids first introduced in \cite{Fa-Tr18} and further studied in \cite{An-Tr18}. In Sect.~\ref{sec:4.3}, we obtain an ``object decomposition theorem'' (Corollary \ref{cor:4.13}) for certain categories with finite products which, among other things, leads to a monoid-theoretic and, in a way, more conceptual proof that every module $M$ of finite uniform dimension over a ring $R$ is a direct sum of finitely many indecomposable $R$-modules.

\subsection{Classical factorizations}
\label{sec:4.1}
Many fundamental aspects of the classical theory of factorization come down to the study of various phenomena that are related to the possibility or impossibility of factoring the $\preceq$-non-units of a monoid $H$ into $\preceq$-atoms or $\preceq$-irreducibles, where $\preceq$ is either the divisibility preorder $\mid_H$, or the ``divides from the left'' preorder $\vdash_H$, or the ``divides from the right'' preorder $\dashv_H$ (Example \ref{exa:3.5}). 
The following corollary will help us to substantiate our claims.

\begin{corollary}
	\label{cor:4.1}
	If $H$ is a Dedekind-finite monoid and the divisibility preorder $\mid_H$ is artinian, then every non-unit of $H$ factors as a \textup{(}non-empty, finite\textup{)} product of $\mid_H$-irreducibles.
\end{corollary}

\begin{proof}
	It is enough to apply Theorem \ref{thm:3.11}, after recalling from Remark \ref{rem:3.7}\ref{rem:3.7(2)} that, if $H$ is Dedekind-finite, then the $\mid_H$-units of $H$ are precisely the units of $H$.
\end{proof}

As simple as it may be, Corollary \ref{cor:4.1} is a non-commutative generalization of Theorem 3.2 in Anderson and Valdes-Leon's seminal paper \cite{AnVL96} on irreducible factorizations in commutative rings: Every commutative monoid is Dedekind-finite and, by Remark \ref{rem:3.7}\ref{rem:3.7(2)}, the $\mid_H$-irreducibles of a Dedekind-finite monoid $H$ are precisely the irreducibles of $H$ as per Anderson and Valdes-Leon's work. 

Actually, we will show in the remainder of this section that Corollary \ref{cor:4.1} is also a generalization of \cite[Proposition 0.9.3]{Co06}, i.e., Cohn's classical result on atomic factorizations in cancellative monoids.

\begin{definition}\label{def:4.2}
A monoid $H$ is \evid{unit-cancellative} if $xy \ne x \ne yx$ for all $x, y \in H$ with $y \notin H^\times$; and is \evid{acyclic} if $uxv \ne x$ for all $u, v, x \in H$ with $u \notin H^\times$ or $v \notin H^\times$.
\end{definition}

Unit-cancellative monoids were recently introduced in \cite{FaGeKaTr17,Fa-Tr18}, as part of a broader program aimed to extend various aspects of the classical theory of factorization to the non-cancellative setting (every cancellative monoid is, obviously, unit-cancellative): One motivation for this is that the non-zero ideals of a commutative noetherian domain form a unit-cancellative monoid when endowed with the usual operation of ideal multiplication, but in general this monoid is not cancellative, see \cite[Sect.~3]{FaGeKaTr17} and \cite[Sect.~4]{GeRe19} for details. Another motivation comes from the monoid-theoretic approach to the study of direct sum decompositions of modules pioneered by A.~Facchini and R.~Wiegand, see \cite{Fa02, WiWi09, BaWi13, BaGe14, Ba-Sm20} and references therein (we will come back to this point in Sect.~\ref{sec:4.3}).

Acyclic monoids, on the other hand, are apparently new in the literature. For one thing, it is obvious that every acyclic monoid is unit-cancellative; all free monoids are acyclic; and a commutative monoid is acyclic if and only if it is unit-cancellative. Note, though, that a non-commutative cancellative monoid need not be acyclic, even if it is finitely generated (Example \ref{exa:4.8}). 

The next proposition and its corollary will help to clarify certain aspects of the arithmetic of unit-cancellative or acyclic monoids that are relevant to the goals of this subsection.

\begin{proposition}\label{prop:4.3}
	Let $H$ be a unit-cancellative monoid. The following hold:
\begin{enumerate}[label=\textup{(\roman{*})}]
\item\label{prop:4.3(i)} $H$ is Dedekind-finite and an element $u \in H$ is a $\vdash_H$-unit (resp., a $\dashv_H$-unit), if and only if $u$ is a $\mid_H$-unit, if and only if $u$ is a unit.
\item\label{prop:4.3(ii)} An element $a \in H$ is a  $\vdash_H$-quark (resp., a $\dashv_H$-quark), if and only if $a$ is a $\vdash_H$-atom (resp., a $\dashv_H$-atom), if and only if $a$ is a $\mid_H$-atom, if and only if $a$ is an atom.
\item\label{prop:4.3(iii)} Every $\mid_H$-atom of $H$ is a $\mid_H$-quark.
\end{enumerate}
\end{proposition}

\begin{proof}
\ref{prop:4.3(i)} By Remark \ref{rem:3.7}\ref{rem:3.7(2)}, it suffices to check that $H$ is Dedekind-finite, and this is straightforward: If $xy = 1_H$ for some $x, y \in H$, then $xyx=x$ and hence $yx \in H^\times$ (by the fact that $H$ is unit-cancellative). 

\vskip 0.05cm

\ref{prop:4.3(ii)} Let $a \in H$. We will only show that $a$ is a $\vdash_H$-quark if and only if $a$ is an atom, since we get from Remark \ref{rem:3.7}\ref{rem:3.7(2)} and part \ref{prop:4.3(i)} that $a$ is a $\vdash_H$-atom if and only if $a$ is a $\mid_H$-atom, if and only if $a$ is an atom: The analogous statement for $\dashv_H$ can be proved in a similar way.

Assume first that $a$ is an atom but not a $\vdash_H$-quark. Then $aH \subsetneq xH$ for some $\vdash_H$-non-unit $x \in H$. In consequence, $x \notin H^\times$ and $a = xu$ for some $u \in H$, which can only happen if $u \in H^\times$ (because $a$ is an atom and $x$ a non-unit). It follows that $a \vdash_H x$ and hence $aH \subsetneq xH \subseteq aH$ (absurd). 

Now, suppose by way of contradiction that $a$ is a $\vdash_H$-quark but not an atom. Then $a = xy$ for some non-units $x, y \in H$; and from here we get that $a \vdash_H x$, since $x$ divides $a$ from the left but, by part \ref{prop:4.3(i)}, is not a $\vdash_H$-unit. Thus $a = xy = auy$ for some $u \in H$, which is impossible for it implies that $uy \in \allowbreak H^\times$ (by the unit-cancellativity of $H$) and hence $y \in H^\times$ (again by part \ref{prop:4.3(i)}).

\vskip 0.05cm

\ref{prop:4.3(iii)} Assume to the contrary that there is a $\mid_H$-atom $a \in H$ that is not a $\mid_H$-quark. There then exists a $\mid_H$-non-unit $b \in H$ such that $b \mid_H a$ and $a \nmid_H b$. In particular, $a = ubv$ for some $u, v \in H$ such that $u$ or $v$ is not a unit. This, however, is only possible if $ub$ or $bv$ is a unit, because $a$ is a $\mid_H$-atom and, by part \ref{prop:4.3(ii)}, every $\mid_H$-atom is an atom. So, by part \ref{prop:4.3(i)}, $b$ is a unit and hence a $\mid_H$-unit (absurd).
\end{proof}

\begin{corollary}\label{cor:4.4}
Let $H$ be an acyclic monoid, and let $a \in H$. The following are equivalent:
\begin{enumerate}[label=\textup{(\alph{*})}]
	\item\label{cor:4.4(a)} $a$ is a $\vdash_H$-quark \textup{(}resp., a $\dashv_H$-quark\textup{)}.
	\item\label{cor:4.4(b)} $a$ is a $\vdash_H$-atom \textup{(}resp., a $\dashv_H$-atom\textup{)}.
	\item\label{cor:4.4(c)} $a$ is a $\mid_H$-quark.
	\item\label{cor:4.4(d)} $a$ is a $\mid_H$-atom.
	\item\label{cor:4.4(e)} $a$ is an atom.
	\item\label{cor:4.4(f)} $a$ is a $\mid_H$-irreducible.
\end{enumerate}
\end{corollary}

\begin{proof}
	An acyclic monoid is also unit-cancellative, so we get from Proposition \ref{prop:4.3} that $H$ is Dedekind-finite, units and $\mid_H$-units are one and the same thing, every $\mid_H$-atom is a $\mid_H$-quark, and conditions \ref{cor:4.4(a)}, \ref{cor:4.4(b)}, \ref{cor:4.4(d)}, and \ref{cor:4.4(e)} are all equivalent. On the other hand, we have from Remark \ref{rem:3.7}\ref{rem:3.7(4)} that every $\mid_H$-quark is a $\mid_H$-irreducible. Therefore, it only remains to see that \ref{cor:4.4(f)} $\Rightarrow$ \ref{cor:4.4(d)}.
	
	To this end, let $a \in H$ be a $\mid_H$-irreducible and suppose for a contradiction that $a$ is not a $\mid_H$-atom. Since every $\mid_H$-unit is a unit, it follows that $a = xy$ for some $x, y \in H \setminus H^\times$. Thus, $x \mid_H a$ and $y \mid_H a$; and since $a$ is a $\mid_H$-irreducible, this is only possible if $a \mid_H x$ or $a \mid_H y$. To wit, $x = uav$ or $y = uav$ for some $u, v \in H$. In consequence, $a = uavy$ or $a = xuav$; and this implies, by the acyclicity of $H$, that $xu \in H^\times$ or $vy \in H^\times$. So, using that $H$ is Dedekind-finite, we get $x \in H^\times$ or $y \in H^\times$ (absurd).
\end{proof}

Based on Proposition \ref{prop:4.3}, Corollary \ref{cor:4.4}, and items \ref{rem:3.7(2)} and \ref{rem:3.7(4)} of Remark \ref{rem:3.7}, the two diagrams below
provide a succinct o\-ver\-view of the logical relations between $\preceq$-quarks, $\preceq$-atoms, and $\preceq$-irreducibles for a unit-cancellative or acyclic monoid $H$, when $\preceq$ is either the ``divides from the left'' preorder $\vdash_H$, the ``divides from the right'' preorder $\dashv_H$, or the divisibility preorder $\mid_H$ (note that the $\mid_H$-ir\-re\-duc\-i\-bles of a cancellative monoid need not be $\mid_H$-atoms or $\mid_H$-quarks, as will be confirmed by Example \ref{exa:4.8}).
{\small
\begin{figure*}[h]
\centering
\begin{minipage}{.5\textwidth}
\centering
\begin{tikzcd}[row sep=0.3cm]
	\text{$\vdash_H$-quark}  
	& 
	\text{$\vdash_H$-atom} 
	\arrow[Rightarrow, shorten <= 1pt, shorten >= 1pt]{r}
	\arrow[Leftrightarrow, shorten <= 1pt, shorten >= 1pt]{l}
	& 
	\text{$\vdash_H$-irreducible} 
	\\
	& 
	&
	\text{$\mid_H$-quark}
	\arrow[Rightarrow, shorten <= 1pt, shorten >= 1pt]{dd}
	\\
	\text{atom}  
	\arrow[Leftrightarrow, shorten <= 1pt, shorten >= 1pt]{r} 
	& \text{$\mid_H$-atom} 
	\arrow[Leftrightarrow, shorten <= 1pt, shorten >= 1pt]{uu} 
	\arrow[Leftrightarrow, shorten <= 1pt, shorten >= 1pt]{dd}
	\arrow[Rightarrow, shorten <= 1pt, shorten >= 1pt]{rd}
	\arrow[Rightarrow, shorten <= 1pt, shorten >= 1pt]{ru} 
	& 
	\\
	& 
	&
	\text{$\mid_H$-irreducible} 
	\\
	\text{$\dashv_H$-quark}  
	\arrow[Leftrightarrow, shorten <= 1pt, shorten >= 1pt]{r} 
	& 
	\text{$\dashv_H$-atom} 
	\arrow[Rightarrow, shorten <= 1pt, shorten >= 1pt]{r} 
	& 
	\text{$\dashv_H$-irreducible}
\end{tikzcd}
\captionsetup{labelformat=empty}
\caption{\small
	(a) The unit-cancellative case.}
\label{fig:right-diagram}
\end{minipage}%
\hfil
\begin{minipage}{.5\textwidth}
\centering
\begin{tikzcd}[row sep=0.3cm]
\text{$\vdash_H$-quark}  
& 
\text{$\vdash_H$-atom} 
\arrow[Rightarrow, shorten <= 1pt, shorten >= 1pt]{r}
\arrow[Leftrightarrow, shorten <= 1pt, shorten >= 1pt]{l}
& 
\text{$\vdash_H$-irreducible} 
\\
& 
&
\text{$\mid_H$-quark}
\\
\text{atom}  
\arrow[Leftrightarrow, shorten <= 1pt, shorten >= 1pt]{r} 
& \text{$\mid_H$-atom} 
\arrow[Leftrightarrow, shorten <= 1pt, shorten >= 1pt]{uu} 
\arrow[Leftrightarrow, shorten <= 1pt, shorten >= 1pt]{dd}
\arrow[Leftrightarrow, shorten <= 1pt, shorten >= 1pt]{rd}
\arrow[Leftrightarrow, shorten <= 1pt, shorten >= 1pt]{ru} 
& 
\\
& 
&
\text{$\mid_H$-irreducible} 
\\
\text{$\dashv_H$-quark}  
\arrow[Leftrightarrow, shorten <= 1pt, shorten >= 1pt]{r} 
& 
\text{$\dashv_H$-atom} 
\arrow[Rightarrow, shorten <= 1pt, shorten >= 1pt]{r} 
& 
\text{$\dashv_H$-irreducible}
\end{tikzcd}
\captionsetup{labelformat=empty}
\caption{\small 
	(b) The acyclic case.}
\end{minipage}
\end{figure*}
}

With this done, we are ready for the next theorem and its corollary: The latter subsumes \cite[The\-o\-rem 2.28(i)]{Fa-Tr18}, i.e., Fan and the author's generalization (to unit-cancellative monoids) of Cohn's \cite[Proposition 0.9.3]{Co06} (see also the comment after Corollary \ref{cor:4.1}).

\begin{theorem}\label{thm:4.5}
The following conditions are equivalent for a monoid $H$:
\begin{enumerate}[label=\textup{(\alph{*})}]
\item\label{thm:4.5(a)} $H$ is unit-cancellative and the preorders $\vdash_H$ and $\dashv_H$ are both artinian.
\item\label{thm:4.5(b)} $H$ is acyclic and the divisibility preorder $\mid_H$ is artinian.
\end{enumerate}
Moreover, each of these conditions implies that every non-unit of $H$ factors as a product of atoms.
\end{theorem}

\begin{proof}
The ``Moreover'' part of the theorem is an obvious consequence of Corollaries \ref{cor:4.1} and \ref{cor:4.4} and the equivalence between Conditions \ref{thm:4.5(a)} and \ref{thm:4.5(b)}, so we focus on proving that \ref{thm:4.5(a)} $\Rightarrow$ \ref{thm:4.5(b)} $\Rightarrow$ \ref{thm:4.5(a)}.

\vskip 0.05cm

\ref{thm:4.5(a)} $\Rightarrow$ \ref{thm:4.5(b)}: We divide the proof into two parts. First, we show that $H$ is an acyclic monoid (\textsc{Part 1}). Next, we verify that $\mid_H$ is an artinian preorder (\textsc{Part 2}).

\vskip 0.05cm

\textsc{Part 1:} \emph{$H$ is acyclic}. Assume by way of contradiction that $H$ is not acyclic. Since the statement to be proved is ``left-right symmetric'', it follows (without loss of generality) that the set 
\[
\Lambda := \{x \in H \colon x = uxv \text{ for some } u, v \in H \text{ with } u \notin H^\times\}
\]
is non-empty; and since $\dashv_H$ is artinian, we get from Remark \ref{rem:3.10}\ref{rem:3.10(2)} that $\Lambda$ has at least one $\dashv_H$-minimal element $\bar{x}$. In particular, $\bar{x}$ is an element of $\Lambda$ and, hence, $\bar{x} = u\bar{x} v$ for some $u, v \in H$ with $u \notin H^\times$. Set $y := \bar{x} v$. Since $uyv =  (u \bar{x} v) \, v = \bar{x} v = y$ and $u \notin H^\times$, we see that $y$ is in $\Lambda$. On the other hand, we have that $\bar{x} = u \bar{x} v = uy$ and hence $y \dashv_H \bar{x}$. Recalling that $\bar{x}$ is a $\dashv_H$-minimal element of $\Lambda$, it follows that $\bar{x}$ is $\dashv_H$-equivalent to $y$, that is, $H\bar{x} = Hy$. This, however, means that $\bar{x} v = y = w \bar{x}$ for some $w \in H$, with the result that $\bar{x} = u \bar{x} v = uw \bar{x}$. So, using that $H$ is unit-cancellative and hence Dedekind-finite (by item \ref{prop:4.3(i)} of Proposition \ref{prop:4.3}), we conclude that $u$ is a unit (absurd). 

\vskip 0.05cm

\textsc{Part 2:} \emph{$\mid_H$ is artinian}. Let $(x_k)_{k \ge 0}$ be a $\mid_H$-non-increasing sequence. We need to find that $x_k \mid_H x_{k+1}$ for all large $k$, and we will actually show the stronger statement that $x_k \in H^\times x_{k+1} H^\times$ from some $k$ on.

To start with, we are given that, for each $k \in \mathbb N^+$, there exist $u_k, v_k \in H$ such that $x_{k-1} = u_k x_k v_k$. We claim that all but finitely many terms of the sequence $v_1, v_2, \ldots$ are units; mutatis mutandis, the same argument also applies to the sequence $u_1, u_2, \ldots,$ and this will be enough to conclude the proof.

Fix $k \in \mathbb N^+$ and set $q_k := u_1 \cdots u_k$. Since $u_k x_k v_k = x_{k-1}$, it is evident that $q_k x_k v_k = q_{k-1} u_k x_k v_k = q_{k-1} x_{k-1}$ and hence $q_k x_k \vdash_H q_{k-1} x_{k-1}$; i.e., the sequence $q_1 x_1, q_2 x_2, \ldots$ is $\vdash_H$-non-increasing. Since $\vdash_H$ is artinian (by hypothesis), it follows that there exists $k' \in \mathbb N^+$ such that, for $k \ge k'$, $q_{k-1} x_{k-1} \vdash_H q_k x_k$ and, hence, $q_k x_k = q_{k-1} x_{k-1} r_{k-1}$ for some $r_{k-1} \in H$. In consequence, we get from the above that, for $k \ge k'$, $
q_{k-1} x_{k-1} = q_k x_k v_k = q_{k-1} x_{k-1} r_{k-1} v_k$. So, recalling that $H$ is unit-cancellative (and Dedekind-finite), we see that, for all large $k \in \mathbb N^+$, $r_{k-1} v_k$ is a unit and hence the same is true of $v_k$ (as wished).

\vskip 0.05cm

\ref{thm:4.5(b)} $\Rightarrow$ \ref{thm:4.5(a)}: Since every acyclic monoid is unit-cancellative (see the comment after Definition \ref{def:4.2}), it is sufficient to show that the preorders $\vdash_H$ and $\dashv_H$ are both artinian: We will work out the details for $\vdash_H$ only, as the other case is essentially the same.

Let $(x_k)_{k \ge 0}$ be a $\vdash_H$-non-increasing sequence, so that, for every $k \in \mathbb N$, there is an element $y_k \in H$ such that $x_k = x_{k+1} y_k$. We need to show that $x_k \vdash_H x_{k+1}$ for all but finitely many $k$, and we will actually prove the stronger statement that $y_k$ is a unit for all large $k$.

Indeed, it is clear from the standing assumptions that $(x_k)_{k \ge 1}$ is a $\mid_H$-non-increasing sequence.
Since $\mid_H$ is artinian, it follows that there exists $k' \in \mathbb N$ such that, for every $k \ge k'$, $x_{k+1} = u_k x_k v_k = u_k x_{k+1} y_k v_k$ for some $u_k, v_k \in H$. By the acyclicity of $H$, this yields that, for all large $k$, $y_k v_k$ is a unit and hence $y_k$ is a unit (recall once again that acyclic monoids are unit-cancellative and hence Dedekind-finite).
\end{proof}

\begin{corollary}\label{cor:4.6}
The following conditions are equivalent for a monoid $H$:

\begin{enumerate}[label=\textup{(\alph{*})}]
\item\label{cor:4.6(a)} $H$ is unit-cancellative and satisfies the \textup{ACCPR} and the \textup{ACCPL}.
\item\label{cor:4.6(b)} $H$ is acyclic and satisfies the \textup{ACCP}.
\end{enumerate}
Moreover, each of these conditions implies that every non-unit of $H$ factors as a product of atoms.
\end{corollary}

\begin{proof}
This is simply a reformulation of Theorem \ref{thm:4.5} based on Remark \ref{rem:3.10}\ref{rem:3.10(3)}.
\end{proof}

Corollary \ref{cor:4.6} is, in a way, best possible, as suggested by the next two examples (see Sect.~\ref{sect:5} for a couple of open questions that could further clarify the picture): The first shows that an acyclic, cancellative monoid satisfying the ACCPL need not satisfy the ACCP; the second shows, among other things, that a reduced, finitely generated, cancellative, non-commutative monoid $H$ can be $\mid_H$-atomic or $\mid_H$-factorable without satisfying the ACCP (cf.~the comments at the end of Sect.~\ref{sec:3}).

\begin{example}\label{exa:4.7}
In \cite[Example 2.6]{Ma-Zi-09}, R.~Mazurek and M.~Ziembowski construct a linearly ordered monoid $(H, \preceq)$ that satisfies the ACCPL but not the ACCPR (``linearly or\-dered'' means that $\preceq$ is a total order such that $uxv \prec uyv$ for all $u, v, x, y \in H$ with $x \prec y$, as seen by unravelling Definition \ref{def:3.4}); moreover, $1_H \preceq x$ for every $x \in H$. It follows that $H$ is cancellative and acyclic (in particular, we have $x \prec uxv$ for all $u, v, x \in H$ with $u \ne 1_H$ or $v \ne 1_H$). Consequently, we get from Corollary \ref{cor:4.6} that $H$ does not satisfy the ACCP, or else it would also satisfy the ACCPR (a contradiction). 
\end{example}

\begin{example}\label{exa:4.8}
Fix $n \in \mathbb N^+$, and let $H$ be the presentation $\mathrm{Mon}\langle A \mid R \rangle$, where $A$ is the $2$-element set $\{a, b\}$ and $R$ is the relation $\{(b^{\ast n}, a * b^{\ast n} * a)\}$ on the monoid $\mathscr F(A)$ (see Sect.~\ref{sec:2.3} for terminology and notation). By Theorem \ref{th:adian-theorem}, $H$ is a cancellative monoid, for it is defined by a finite pre\-sen\-tation whose left and right graphs are cycle-free (each is a path graph on two vertices). 
In addition, it is clear that, if two $A$-words $\mathfrak u$ and $\mathfrak v$ are congruent modulo $R^\sharp$, then $
\mathsf{v}_b^A(\mathfrak u) = \mathsf{v}_b^A(\mathfrak v)$; whence the only unit of $H$ is the identity $1_H$, i.e., the congruence class of $\varepsilon_A$ (the empty $A$-word) mod\-ulo $R^\sharp$. 

Thus, we obtain from Proposition \ref{prop:4.3}\ref{prop:4.3(i)} that $1_H$ is also the only $\mid_H$-unit of $H$.
It follows, by Remark \ref{rem:3.7}\ref{rem:3.7(2)} and the same ``$b$-adic argument'' used above, that $a$ is a $\mid_H$-atom; and so is $b$ for $n \ge 2$. On the other hand, every $\mid_H$-atom is a $\mid_H$-irreducible (Remark \ref{rem:3.7}\ref{rem:3.7(4)}); and we aim to show that
\begin{itemize}
\item if $n = 1$, then $b$ is a $\mid_H$-irreducible (\textsc{Part 1}) but neither a $\mid_H$-atom nor a $\mid_H$-quark (\textsc{Part 2});
\item none of the preorders $\vdash_H$, $\dashv_H$, and $\mid_H$ is artinian (\textsc{Part 3}).
\end{itemize}
Overall, this will mean that $H$ is $\mid_H$-factorable for all values of $n$, and is $\mid_H$-atomic if and only if $n \ge 2$. (Note that these conclusions cannot be drawn from Theorem \ref{thm:4.5}: The reason is that $H$ is not acyclic, because $b^{\ast n} \equiv a \ast b^{\ast n} \ast a \bmod R^\sharp$ and $a \notin H^\times$.)
\vskip 0.05cm

\textsc{Part 1:} \emph{$b$ is $\mid_H$-irreducible}. If not, then we see from Remark \ref{rem:3.7}\ref{rem:3.7(2)} that $b \equiv \mathfrak u \ast \mathfrak v \bmod R^\sharp$ for some $\mathfrak u, \mathfrak v \in \mathscr F(A)$ such that $b \nmid_H \mathfrak u$ and $b \nmid_H \mathfrak v$ (recall that the only $\mid_H$-unit of $H$ is $1_H$). So, $\mathfrak u$ and $\mathfrak v$ are powers of $a$ in $\mathscr F(A)$; whence $b \equiv a^{\ast k} \bmod R^\sharp$ for some $k \in \mathbb N$. But this is a contradiction, as we know from the above that two $A$-words can only be congruent modulo $R^\sharp$ if they contain an equal number of $b$'s.

\vskip 0.05cm

\textsc{Part 2:} \emph{If $n = 1$, then $b$ is neither a $\mid_H$-atom nor a $\mid_H$-quark}. If $n = 1$, then $b \equiv a \ast b \ast a \allowbreak \bmod R^\sharp$ (by the definition of $H$), but neither $a$ nor $b \ast a$ is a $\mid_H$-unit; therefore, $b$ is a $\mid_H$-non-unit and factors in $H$ as a product of two $\mid_H$-non-units; viz., $b$ is not a $\mid_H$-atom. Moreover, it is clear that 
$b \nmid_H a$, or else $0 = \mathsf{v}_b^A(a) = \mathsf{v}_b^A(b) = 1$ (absurd); so, $b$ is not a $\mid_H$-quark, because $a \mid_H b$.

\vskip 0.05cm

\textsc{Part 3:} \emph{None of $\vdash_H$, $\dashv_H$, and $\mid_H$ is artinian}. Let $(\mathfrak z_k)_{k \ge 0}$ be the $H$-valued sequence whose $k^\mathrm{th}$ term $\mathfrak z_k$ is the $A$-word $b^{\ast n} \ast a^{\ast k} \ast b^{\ast n}$ (taken modulo $R^\sharp$). For every $k \in \mathbb N$, we have that
\[
\mathfrak z_{k+1} \ast a \equiv b^{\ast n} \ast a^{\ast k} \ast a \ast b^{\ast n} \ast a \equiv \allowbreak b^{\ast n} \ast a^{\ast k} \ast \allowbreak b^{\ast n} \equiv \mathfrak z_k \bmod R^\sharp;
\]
and in a similar way, $a \ast \mathfrak z_{k+1} \equiv \mathfrak z_k \bmod R^\sharp$. To wit, the sequence $(\mathfrak z_k)_{k \ge 0}$ is $\vdash_H$- and $\dashv_H$-non-increasing; therefore, it is also $\mid_H$-non-increasing (by the fact that $\vdash_H$ is a subrelation of $\mid_H$).

Suppose for a contradiction that there exists $k \in \mathbb N$ such that $\mathfrak z_i \vdash_H \mathfrak z_{k+1}$, $\mathfrak z_i \dashv_H \mathfrak z_{k+1}$, or $\mathfrak z_i \mid_H \mathfrak z_{k+1}$ for some $i \in \llbracket 0, k \rrbracket$. In either case, we can find $\mathfrak u, \mathfrak v \in \mathscr F(A)$ such that
$\mathfrak u \ast \mathfrak z_i \ast \mathfrak v \equiv \allowbreak \mathfrak z_{k+1} \bmod R^\sharp$; and we may assume that $i$ is the smallest integer between $0$ and $k$ (inclusive) for which this holds.

Now, since $\mathsf{v}_b^A(\mathfrak z_i) = \allowbreak \mathsf{v}_b^A(\mathfrak z_{k+1}) = \allowbreak 2n$, $\mathfrak u$ and $\mathfrak v$ are necessarily powers of $a$ in $\mathscr F(A)$, that is, $\mathfrak u = a^{\ast r}$ and $\mathfrak v = a^{\ast s}$ for some $r, s \in \mathbb N$ (recall that if two $A$-words are congruent modulo $R^\sharp$ then they have the same $b$-adic valuation):
We claim that neither $r$ nor $s$ can be zero. 
For, assume to the contrary that $r = 0$ (the other case is symmetric). It then follows from the above that 
\[
b^{\ast n} \ast a^{\ast i} \ast b^{\ast n} \ast a^{\ast s} \equiv b^{\ast n} \ast \allowbreak a^{\ast (k+1)} \ast b^{\ast n} \bmod R^\sharp;
\] 
and by the cancellativity of $H$, we get 
\begin{equation}\label{equ:1}
b^{\ast n} \ast a^{\ast s} \equiv a^{\ast (k+1 - i)} \ast b^{\ast n} \bmod R^\sharp.
\end{equation}
But it is readily checked (by induction) that 
\begin{equation}\label{equ:2}
a^{\ast j} \ast b^{\ast n} \ast a^{\ast j} \equiv b^{\ast n} \bmod R^\sharp,
\qquad \text{for every }
j \in \mathbb N.
\end{equation}
So, multiplying both sides of the congruence in Eq.~\eqref{equ:1} by $a^{\ast s}$, we obtain that
\[
a^{\ast (s+k+1 - i)} \ast b^{\ast n} 
\equiv a^{\ast s} \ast b^{\ast n} \ast a^{\ast s} \stackrel{\eqref{equ:2}}{\equiv} b^{\ast n} \bmod R^\sharp; 
\]
whence $a^{\ast (s+k+1-i)} \equiv \varepsilon_A \bmod R^\sharp$ (by the fact that $H$ is cancellative). This however is impossible, since it gives $i = s + k + 1 \ge k+1$. In consequence, $r$ and $s$ must be non-zero (as wished).

To sum it up, we have established that $a^{\ast r} \ast \mathfrak z_i \ast a^{\ast s} \equiv \mathfrak z_{k+1} \bmod R^\sharp$ for some $r, s \in \mathbb N^+$; and we shall see from here that $i = 0$. In fact, assume that $i$ is a positive integer. Then
\[
\mathfrak z_{k+1} \equiv a^{\ast r} \ast b^{\ast n} \ast a^{\ast i} \ast b^{\ast n} \ast a^{\ast s} \equiv a^{\ast (r-1)} \ast b^{\ast n} \ast a^{\ast (i-1)} \ast b^{\ast n} \ast a^{\ast s} \equiv a^{\ast (r-1)} \ast \mathfrak z_{i-1} \ast a^{\ast s} \bmod R^\sharp;
\]
and similarly, $\mathfrak z_{k+1} \equiv a^{\ast r} \ast \mathfrak z_{i-1} \ast a^{\ast (s-1)} \bmod R^\sharp$. Thus $\mathfrak z_{i-1} \vdash_H \mathfrak z_{k+1}$ and $\mathfrak z_{i-1} \dashv_H \mathfrak z_{k+1}$, contradicting the minimality of $i$. It follows that $i = 0$ (as wished) and hence 
\[
a^{\ast r} \ast {\underbrace{b^{\ast n} \ast a^{\ast 0} \ast b^{\ast n}}_{= \, \mathfrak z_0}} \ast a^{\ast s}
\equiv 
{\underbrace{b^{\ast n} \ast a^{\ast (k+1)} \ast b^{\ast n}}_{= \, \mathfrak z_{k+1}}}
\stackrel{\eqref{equ:2}}{\equiv}  
a^{\ast r} \ast b^{\ast n} \ast a^{\ast (r+s+k+1)} \ast b^{\ast n} \ast a^{\ast s} \bmod R^\sharp;
\]
which, by cancellativity, implies $a^{\ast (r+s+k+1)} \equiv \varepsilon_A \bmod R^\sharp$. But this can only happen if $r+s+k+1 = 0$ (absurd), because there is no non-empty $A$-word congruent to $\varepsilon_A$ modulo $R^\sharp$ (recall that $H$ is reduced).

So, putting it all together, we conclude that none of the preorders $\vdash_H$, $\dashv_H$, and $\mid_H$ is artinian, since the sequence $(\mathfrak z_k)_{k \ge 0}$ is (strictly) decreasing with respect to each of them.
\end{example}

We finish the subsection with a few remarks on Dedekind-finiteness, motivated by the critical role this condition plays in Corollary \ref{cor:4.1} and Theorem \ref{thm:4.5} (see also Proposition \ref{prop:4.11} and Theorem \ref{thm:4.12}).
\begin{remarks}\label{rem:4.9}
	\begin{enumerate*}[resume,label=\textup{(\arabic{*})}]
		\item\label{rem:4.9(2)} Let $H$ be a monoid. We will prove that, if the ``divides from the left'' preorder $\vdash_H$ or the ``divides from the right'' preorder $\dashv_H$ is noetherian, then $H$ is Dedekind-finite. 
		\\
		
		\indent{}In fact, pick $x, y \in H$ such that $xy = 1_H$ and assume $\vdash_H$ is noetherian (the other case is symmetric); it suffices to prove that $y$ is a unit. Since $1_H \vdash_H y \allowbreak  \vdash_H y^2 \vdash_H \cdots$ and, by hypothesis, no in\-fi\-nite se\-quence of elements of $H$ can be (strictly) $\vdash_H$-in\-creas\-ing, we have that $y^{k+1} \vdash_H y^k$ for some $k \in \mathbb N$. It follows that there exists $u \in H$ such that $x^k y^{k+1} u = x^k y^k$; and this, in turn, gives that $xy = 1_H = yu$, for it is immediate (by induction on $k$) that $x^k y^k = 1_H$. So, $y$ is a unit and we are done.
	\end{enumerate*}
	
	\vskip 0.05cm
	
	\begin{enumerate*}[resume,label=\textup{(\arabic{*})}]
		\item\label{rem:4.9(5)} As a complement to the conclusions made in item \ref{rem:4.9(2)}, we will show that neither the artinianity nor the noetherianity of the divisibility preorder $\mid_H$ is a sufficient condition for a monoid $H$ to be Dedekind-finite; nor is the artinianity of $\vdash_H$ or $\dashv_H$.\\
		
		\indent{}Indeed, let $M$ be a monoid which is not Dedekind-finite, in such a way that
		we can pick $x, y \in M$ with $xy = 1_M \ne yx$. Accordingly, let $H$ be the submonoid of $M$ generated by the set $\{x,y\}$. Of course, $H$ is not Dedekind-finite. However, $H = HzH$ for all $z \in H$, and hence $\mid_H$ is artinian and noetherian.\\
		
		\indent{}In fact, fix $z \in H$. Then $z \in \{x, y\}^n$ for some $n \in \mathbb N^+$; and since $x^k y^k = 1_H$ for all $k \in \mathbb N$ (by induction on $k$), it is readily found (by induction on $n$) that there exist $a, b \in \mathbb N$ such that $z = y^a x^b$. It follows that $x^a z y^b = 1_H$ and hence $H \supseteq HzH \supseteq H x^a z y^b H = H$. To wit, $H = HzH$ (as wished).
	\end{enumerate*}

		\indent{}It remains to see that $\vdash_H$ is artinian (the other case is symmetric). To start with, it is easily checked that, if $y^p x^q u = y^r x^s$ for some $p, q, r, s \in \mathbb N$ and $u \in H$, then $p \le r$: Otherwise, we would have
		\[
		1_H = x^r y^r x^s y^s = x^r y^p x^q u\, y^s = y^{p-r} x^q u\, y^s = yv, 
		\]
		and hence $y \in H^\times$ (absurd), where $v := y^{p-r-1} x^q u\, y^s \in H$ (recall that $x^k y^k = 1_H$ for all $k \in \mathbb N$). It follows that, if $(z_k)_{k \ge 0}$ is a $\vdash_H$-non-increasing sequence of elements of $H$, then there exist $\alpha, b_0, b_1, \ldots \in \mathbb N$ such that $z_k = y^\alpha x^{b_k}$ for every large $k \in \bf N$ (we proved above that each $z \in H$ has the form $y^a x^b$ for some $a, b \in \mathbb N$). So $H$ is $\vdash_H$-artinian, because $y^q x^r y^r x^s = y^q x^s$ (i.e., $y^q x^r \vdash_H y^q x^s$) for all $q, r, s \in \bf N$.
	
	\vskip 0.05cm
	
	\begin{enumerate*}[resume,label=\textup{(\arabic{*})}]
		\item\label{rem:4.9(3)} 
		Let $H$ be a \evid{periodic} monoid, meaning that, for each $z \in H$, the subsemigroup of $H$ generated by $\{z\}$ is finite. We claim that $H$ is Dedekind-finite. Indeed, suppose $xy = 1_H$ for some $x, y \in H$; it suffices to show that $y$ is a unit. To this end, note that, since the set $\{y, y^2, \ldots\}$ is finite, we are guaranteed (by the Pigeonhole Principle) that $y^m = y^{m+n}$ for some $m, n \in \mathbb N^+$. Thus, we find that $1_H = x^m y^m = x^m y^{m+n} \allowbreak = y^n$, because $x^k y^k = 1_H$ for all $k \in \mathbb N$ (cf.~item \ref{rem:4.9(2)}). Therefore, $y$ is a unit (as wished).
	\end{enumerate*}
\end{remarks}

\subsection{Power monoids} 
\label{sec:4.2}
Let $H$ be a monoid. Following \cite{An-Tr18}, we let the \evid{reduced power monoid} of $H$, hereafter denoted by $\mathcal P_{\mathrm{fin},1}(H)$, be the monoid obtained by endowing the set of all finite subsets of $H$ containing the identity $1_H$ with the operation of setwise multiplication induced by $H$, so that $
XY = \{xy \colon x \in X,\, y \in Y\}$ for all
$X, Y \in \mathcal P_{\mathrm{fin},1}(H)$. Note that the identity of $\mathcal P_{\mathrm{fin},1}(H)$ is the singleton $\{1_H\}$.

The arithmetic of $\mathcal P_{\mathrm{fin},1}(H)$ is rich and, in a way, rather intricate, even in the fundamental case where $H$ is a cyclic group: Part of the reason lies in the ``highly non-cancellative'' nature of the operation of set\-wise multiplication, which results in a variety of algebraic and arithmetical phenomena not observable in the ``nearly cancellative'' scenarios discussed in Sect.~\ref{sec:4.1} (see \cite{Fa-Tr18,An-Tr18} for further details). 

Below we add to this line of research by showing that $\mathcal{P}_{\mathrm{fin}, 1}(H)$ is $\mid_{\mathcal{P}_{\mathrm{fin},1}(H)}$-factorable, and by characterizing the monoids $H$ for which every $X \in \mathcal{P}_{\mathrm{fin}, 1}(H)$ factors as a product of ``ordinary atoms'' (Proposition \ref{prop:4.11}\ref{prop:4.11(ii)} and 
Theorem \ref{thm:4.12}).
In the proofs, we will freely use that $|XY| \ge \max(|X|, |Y|)$ for all $
X, Y \allowbreak \in \allowbreak \mathcal P_{\mathrm{fin},1}(H)$, and that $X \subseteq Y$ when\-ever $X \mid_{\mathcal P_{\mathrm{fin},1}(H)} Y$ (by the fact that each set in $\mathcal P_{\mathrm{fin},1}(H)$ contains $1_H$).

We start with a slight refinement of \cite[Lemma 3.8]{An-Tr18}. We recall that an element $x$ in a monoid is an \evid{i\-dem\-po\-tent} if $x^2 = x$; and is a \evid{proper i\-dem\-po\-tent} if $x$ is an idempotent but not the identity.

\begin{lemma}\label{lem:4.10}
	Let $H$ be a monoid with no proper idempotent. Then $H$ is Dedekind-finite, and the sub\-sem\-i\-group generated by any non-unit of $H$ is infinite.
\end{lemma}
\begin{proof}
	First, suppose that $yz = 1_H$ for some $y, z \in H$. Then $(zy)^2 = z(yz)y = zy$; and since $H$ has no proper i\-dem\-po\-tents, we conclude that $zy=1_H$. Consequently, $H$ is Dedekind-finite.
	
	Next, assume $\{x, x^2, \ldots\}$ is a finite subset of $H$ for some $x \in H$. By the Pigeonhole Principle, there exist $n, k \in \mathbb N^+$ such that $x^n = x^{n+k}$; and by a routine induction, this implies that $x^n = x^{n+hk}$ for all $h \in \mathbb N$. So, we obtain that $
	(x^{nk})^2 = x^{2nk} = x^{(k+1)n}x^{(k-1)n} = x^n x^{(k-1)n} = x^{nk}$. Since $H$ has no proper idempotents, it follows that $x^{nk}=1_H$. To wit, $x$ is a unit.
\end{proof}

\begin{proposition}\label{prop:4.11}
	Let $H$ be a monoid.
	The following hold:
	\begin{enumerate}[label=\textup{(\roman{*})}]
	\item\label{prop:4.11(i)} $\mathcal{P}_{\mathrm{fin},1}(H)$ is a reduced and Dedekind-finite monoid, and the preorder $\mid_{\mathcal{P}_{\mathrm{fin},1}(H)}$ is artinian.
	\item\label{prop:4.11(ii)} Every $X \in \mathcal{P}_{\mathrm{fin},1}(H)$ factors as a product of $\mid_{\mathcal{P}_{\mathrm{fin},1}(H)}$-irreducibles.
	\item\label{prop:4.11(iii)} A set $A \in \mathcal P_{\mathrm{fin},1}(H)$ is a $\mid_{\mathcal{P}_{\mathrm{fin},1}(H)}$-irreducible if and only if it is a $\mid_{\mathcal{P}_{\mathrm{fin},1}(H)}$-quark.
\end{enumerate}
\end{proposition}

\begin{proof} 
\ref{prop:4.11(ii)} is straightforward from \ref{prop:4.11(i)} and Cor\-ol\-lar\-y \ref{cor:4.1} (in particular, note that the identity of $\mathcal P_{\mathrm{fin},1}(H)$ is an empty product of $\mid_{\mathcal P_{\mathrm{fin},1}(H)}$-irreducibles), so we will focus attention on \ref{prop:4.11(i)} and \ref{prop:4.11(iii)}. 
	
\vskip 0.05cm
	
\ref{prop:4.11(i)} Since $|XY| \ge \max(|X|, |Y|)$ for all $X, Y \in \mathcal P_{\mathrm{fin},1}(H)$, it is clear that $XY = \{1_H\}$ if and only if $X$ and $Y$ are singletons, if and only if $X = Y = \{1_H\}$. Thus $\mathcal{P}_{\mathrm{fin},1}(H)$ is reduced and Dedekind-finite. 

On the other hand, since $X$ is contained in $Y$ whenever $X \mid_{\mathcal P_{\mathrm{fin},1}(H)} Y$, it is evident that a $\mid$-non-in\-creas\-ing sequence $(X_k)_{k \ge 0}$ of sets in $\mathcal P_{\mathrm{fin},1}(H)$ is also non-increasing with respect to inclusion; and this, in turn, can only happen if $X_{k+1} = X_k$ for all large $k \in \mathbb N$, because the el\-e\-ments of $\mathcal P_{\mathrm{fin},1}(H)$ are \emph{finite} sets. In consequence, $\mid_{\mathcal P_{\mathrm{fin},1}(H)}$ is an artinian preorder.

	\vskip 0.05cm
	
	\ref{prop:4.11(iii)} It suffices to prove the ``only if'' direction, cf.~Remark \ref{rem:3.7}\ref{rem:3.7(4)}. For, suppose by way of contradiction that there is a $\mid_{\mathcal P_{\mathrm{fin},1}(H)}$-irreducible $A$ of $\mathcal P_{\mathrm{fin},1}(H)$ which is not a $\mid_{\mathcal P_{\mathrm{fin},1}(H)}$-quark. Since $X \subseteq Y$ whenever $X \mid_{\mathcal P_{\mathrm{fin},1}(H)} Y$ and, by part \ref{prop:4.11(i)}, $\mathcal P_{\mathrm{fin},1}(H)$ is reduced and Dedekind-finite, there then exists $B \in \mathcal P_{\mathrm{fin},1}(H)$ such that $B \mid_{\mathcal P_{\mathrm{fin},1}(H)} A$ and $\{1_H\} \subsetneq B \subsetneq A$. It follows that $A = UBV$ for some $U, V \in \mathcal P_{\mathrm{fin},1}(H)$ with $U \ne \allowbreak \{1_H\}$ or $V \ne \{1_H\}$, and this is only possible if $A = UB$ or $A = BV$: Otherwise, each of $U$, $V$, $UB$, and $BV$ is a proper sub\-set of $A$, because $U \subseteq UB \subseteq A$ and $V \subseteq BV \subseteq A$; therefore, $A = XY$ for some $\mid_{\mathcal P_{\mathrm{fin},1}(H)}$-non-units $X$ and $Y$ such that $A \nmid_{\mathcal P_{\mathrm{fin},1}(H)} X$ and $A \nmid_{\mathcal P_{\mathrm{fin},1}(H)} Y$ (where $X := U$ and $Y := BV$ if $U \ne \allowbreak \{1_H\}$, and $X := UA$ and $Y := V$ if $V \ne \allowbreak \{1_H\}$), contradicting that $A$ is a $\mid_{\mathcal P_{\mathrm{fin},1}(H)}$-irreducible.
	
	So, assume $A = UB$ (the other case is similar). Then $U \ne \{1_H\}$ (because $B \subsetneq A$) and hence $U = \allowbreak A$: Otherwise, $UB$ is a factorization of $A$ into two $\mid_{\mathcal P_{\mathrm{fin},1}(H)}$-non-units with $A \nmid_{\mathcal P_{\mathrm{fin},1}(H)} U$ and $A \nmid_{\mathcal P_{\mathrm{fin},1}(H)} \allowbreak B$, again in contradiction to the $\mid_{\mathcal P_{\mathrm{fin},1}(H)}$-irreducibility of $A$. As a result, we have
	\[
	\{1_H\} \subsetneq \allowbreak B \subsetneq \allowbreak AB = \allowbreak A.
	\]
	Ac\-cord\-ing\-ly, pick an element $b \in B \setminus \{1_H \} \subseteq A$ and set $A_b := A \setminus \{b\}$. Then $2 \le |A_b | < |A|$ (note that $\{1_H, b\} \subseteq \allowbreak B \subsetneq \allowbreak A$); and since $1_H \in A_b \cap B$, it is readily seen that
	\[
	A_b B \subseteq AB = A = A_b \cup \{b\} \subseteq A_b B \cup \{b\} \subseteq A_b B \cup B = A_b B,
	\]
	namely, $A = A_b B$. But yet again, this means that $A$ is not a $\mid_{\mathcal P_{\mathrm{fin},1}(H)}$-irreducible (absurd).
\end{proof}

The next result is a sensible refinement of \cite[Theorem 3.9]{An-Tr18}, where it is proved that, for a monoid $H$, every set in $\mathcal P_{\mathrm{fin},1}(H)$ factors as a product of atoms if and only if $1_H \ne x^2 \ne x$ for all $x \in H \setminus \{1_H\}$: The key dif\-fer\-ence is that, by Proposition \ref{prop:4.11}, we here already know that every set in $\mathcal P_{\mathrm{fin},1}(H)$ factors as a product of $\mid_{\mathcal P_{\mathrm{fin},1}(H)}$-quarks (regardless of any condition on $H$).

\begin{theorem}\label{thm:4.12}
The following are equivalent for a monoid $H$:
	
	\begin{enumerate}[label=\textup{(\alph{*})}]
		\item\label{thm:4.12(a)} $1_H \ne x^2 \ne x$ for each $x \in H \setminus \{1_H\}$.
		\item\label{thm:4.12(b)} Each $\mid_{\mathcal P_{\mathrm{fin},1}(H)}$-quark of $\mathcal P_{\mathrm{fin}, 1}(H)$ is an atom, and vice versa.
		\item\label{thm:4.12(c)} Every $X \in \mathcal P_{\mathrm{fin},1}(H)$ factors as a product of atoms.
	\end{enumerate}
\end{theorem}

\begin{proof}
	The implication \ref{thm:4.12(b)} $\Rightarrow$ \ref{thm:4.12(c)} is a trivial consequence of parts \ref{prop:4.11(ii)} and \ref{prop:4.11(iii)} of Prop\-o\-si\-tion \ref{prop:4.11}. So we will concentrate on proving that \ref{thm:4.12(a)} $\Rightarrow$ \ref{thm:4.12(b)} and \ref{thm:4.12(c)} $\Rightarrow$ \ref{thm:4.12(a)}.
	
	\vskip 0.05cm
	
	\ref{thm:4.12(a)} $\Rightarrow$ \ref{thm:4.12(b)}: Suppose for a contradiction that there is a $\mid_{\mathcal P_{\mathrm{fin}, 1}(H)}$-quark $A \in \mathcal P_{\mathrm{fin}, 1}(H)$ which is not an atom.
	Then $A = XY$ for some non-units $X, Y \allowbreak \in \mathcal P_{\mathrm{fin},1}(H)$, which gives by Prop\-o\-si\-tion \ref{prop:4.11}\ref{prop:4.11(i)} that each of $X$ and $Y$ is a $\mid_{\mathcal P_{\mathrm{fin}, 1}(H)}$-non-unit dividing $A$. But this can only happen if $A$, in turn, divides each of $X$ and $Y$,  because $A$ is a $\mid_{\mathcal P_{\mathrm{fin}, 1}(H)}$-quark. So, using that, in $\mathcal P_{\mathrm{fin}, 1}(H)$, ``to divide'' implies ``to be contained'', we conclude that $X = Y = A$ and hence $A = XY = A^2$. It follows (by a routine induction) that $A = A^n$ for every $n \in \mathbb N^+$. In consequence, it is clear that $A = \bigcup_{n \ge 1} A^n$.
		
	Now, pick $a \in A \setminus \{1_H\}$. The subsemigroup of $H$ generated by $\{a\}$ is finite, as we have that $|A| < \infty$ and $\{a, a^2, \ldots\} \subseteq \bigcup_{n \ge 1} A^n = A$. Since $1_H \ne x^2 \ne x$ for each $x \in H \setminus \{1_H\}$ (by hypothesis), we are thus guar\-an\-teed by Lemma \ref{lem:4.10} that $a$ is a unit of $H$ and there exists a smallest integer $n \ge 2$ such that $a^{n+1} = \allowbreak 1_H$. So, setting $B := A \setminus \{a^n\}$ and considering that $A = A^{n+1}$ (as shown above), we obtain 
	\[
	\{1_H, a\} \subseteq B \subsetneq A \subseteq AB \subseteq A^2 = AB \cup A a^n \subseteq AB \cup A^{n+1} = AB \cup A = AB.
	\]	
	Then $A = A^2 = AB$ and hence $B \mid_{\mathcal P_{\mathrm{fin}, 1}(H)} A$. But this contradicts that $A$ is a $\mid_{\mathcal P_{\mathrm{fin}, 1}(H)}$-quark, because $\{1_H\} \subsetneq B \subsetneq A$ and hence $A \nmid_{\mathcal P_{\mathrm{fin}, 1}(H)} B$  (recall that the only $\mid_{\mathcal P_{\mathrm{fin}, 1}(H)}$-unit of $\mathcal P_{\mathrm{fin}, 1}(H)$ is the identity).
	
	Every $\mid_{\mathcal P_{\mathrm{fin}, 1}(H)}$-quark is therefore an atom. For the converse, we have from Prop\-o\-si\-tion \ref{prop:4.11}\ref{prop:4.11(i)} and Remark \ref{rem:3.7}\ref{rem:3.7(2)} that every atom of $\mathcal P_{\mathrm{fin}, 1}(H)$ is a $\mid_{\mathcal P_{\mathrm{fin}, 1}(H)}$-atom; and from Remark \ref{rem:3.7}\ref{rem:3.7(4)} and Prop\-o\-si\-tion \ref{prop:4.11}\ref{prop:4.11(iii)} that every $\mid_{\mathcal P_{\mathrm{fin}, 1}(H)}$-atom is a $\mid_{\mathcal P_{\mathrm{fin}, 1}(H)}$-quark. So, every atom is a $\mid_{\mathcal P_{\mathrm{fin}, 1}(H)}$-quark.
	
	\vskip 0.05cm
	
	\ref{thm:4.12(c)} $\Rightarrow$ \ref{thm:4.12(a)}: Assume to the contrary that there exists an element $x \in H \setminus \{1_H\}$ with $x^2 = 1_H$ or $x^2 = x$. Since $\mathcal P_{\mathrm{fin},1}(H)$ is a reduced  monoid, $\{1_H, x\}$ is by hypothesis a (non-empty) product $A_1 \cdots A_n$ of atoms $A_1, \ldots, A_n \in \mathcal P_{\mathrm{fin},1}(H)$. It follows that $A_i = \{1_H, x\}$ for each $i \in \llb 1, n \rrb$, because $\{1_H\} \subsetneq A_i \subseteq \{1_H, x\}$. This however contradicts that $A_i$ is an atom, by the fact that $\{1_H, x\} = \{1_H, x, x^2\} = \{1_H, x\}^2$.
\end{proof}

It is perhaps worth remarking that there is no obvious way to derive Proposition \ref{prop:4.11}\ref{prop:4.11(iii)} from Corollary \ref{cor:4.4}, or Theorem \ref{thm:4.12} from Proposition \ref{prop:4.3} and Theorem \ref{thm:4.5}: The reason is that, in general, the monoid $\mathcal P_{\mathrm{fin},1}(H)$ is far from being unit-cancellative (for instance, it is clear that $HX = XH = H$ for every $X \subseteq \allowbreak H$, implying that $\mathcal P_{\mathrm{fin},1}(H)$ is unit-cancellative only if $H \notin \mathcal P_{\mathrm{fin},1}(H)$, that is, $|H| = \infty$).

\subsection{Categories and ``object decompositions''}
\label{sec:4.3}
We set out with a quick review of some basic aspects of cat\-e\-go\-ry theory we will need below: We refer the reader to \cite{MacLane98} for all terms used herein without def\-i\-ni\-tion, and we recall from Sect.~\ref{sec:2.1} that we choose Tarski-Grothendieck set theory as a foundation; in particular, we will assume that the objects and the morphisms of the categories we are going to consider all belong to a fixed Grothendieck universe $\mathscr U$.

Let $\mathcal C$ be a category. We denote by $\mathrm{Ob}(\mathcal C)$ and $\mathrm{Arr}(\mathcal C)$, resp., the class of objects and the class of arrows (or morphisms) of $\mathcal C$; and given $A, B \in \mathrm{Ob}(\mathcal C)$, we use $\mathrm{Arr}_\mathcal{C}(A, B)$ for the class of all arrows $f \in \mathrm{Arr}(\mathcal C)$ with domain $A$ and codomain $B$. As usual, an object $T \in \mathrm{Ob}(\mathcal C)$ is \textsf{terminal} if $\mathrm{Arr}_\mathcal{C}(A, T)$ is a singleton for each $A \in \mathrm{Ob}(\mathcal C)$; and an object $P \in \mathrm{Ob}(\mathcal C)$ is a \textsf{product} of an indexed set $(A_i)_{i \in I}$ of objects of $\mathcal C$ if, for each $i \in I$, there is an arrow $p_i \in \mathrm{Arr}_\mathcal{C}(P, A_i)$ for which the following universal property holds: 

\begingroup
\addtolength\leftmargini{-0.1cm}
\begin{quote}
However we choose an object $Q \in \mathrm{Ob}(\mathcal C)$ and an indexed family $(q_i \colon Q \to A_i)_{i \in I}$ of arrows of $\mathcal C$, there is a unique $u \in \mathrm{Arr}_\mathcal{C}(Q, P)$ such that $q_i = p_i \circ_\mathcal{C} u$ for each $i \in I$, where we write $g \circ_\mathcal{C} f$ for the composite of a pair $(f,g) \in \mathrm{Arr}(\mathcal C) \times \mathrm{Arr}(\mathcal C)$ such that the codomain of $f$ is the same as the domain of $g$. 
\end{quote}
\endgroup
\noindent{}It is an elementary fact that a product, when it exists, is unique up to isomorphism; and that an emp\-ty product is nothing else than a terminal object (see \cite[Sect.~III.4]{MacLane98} for further details).

Suppose now that $\mathcal C$ is a category \evid{with finite products}, meaning that every set of objects of $\mathcal C$ indexed by a finite set has a product in $\mathcal C$: By \cite[Sect.~III.5, Proposition 1]{MacLane98}, this is equivalent to requiring that $\mathcal C$ has a terminal object and each pair $(A, B)$ of objects of $\mathcal C$ has a product in $\mathcal C$. We denote by $\mathcal V(\mathcal C)$ the quo\-tient of $\mathrm{Ob}(\mathcal C)$ by the equivalence relation that identifies two objects $A$ and $B$ of $\mathcal C$ if and only if there is an i\-so\-mor\-phism $u \in \mathrm{Arr}_\mathcal{C}(A, B)$; and we call an equivalence class in $\mathcal V(\mathcal C)$ an \textsf{isomorphism class} of $\mathcal C$. Accordingly, we can construct a monoid out of the objects of $\mathcal C$ by endowing the quotient $\mathcal V(\mathcal C)$ with the (binary) operation that maps a pair $(\alpha, \beta)$ of isomorphism classes of $\mathcal C$ to the isomorphism class of a product $A \myprod B \in \mathrm{Ob}(\mathcal C)$ of an object $A \in \alpha$ by an object $B \in \beta$: The operation is well defined by the universal property of products and makes the class $\mathcal V(\mathcal C)$ into a reduced, commutative monoid (see, e.g., \cite[Lemma 1.17]{Fa19}), herein referred to as the \textsf{direct monoid of isomorphism classes} of $\mathcal C$ and, by abuse of notation, identified with $\mathcal V(\mathcal C)$. This leads to the following:

\begin{corollary}\label{cor:4.13}
Let $\mathcal C$ be a category with finite products and assume there exists a function $\lambda: \mathrm{Ob}(\mathcal C) \to \mathbb N$ such that, for all $A, B \in \mathrm{Ob}(\mathcal C)$, the following hold:
\begin{enumerate}[label=\textup{(\arabic{*})}]
\item\label{cor:4.13(1)} $\lambda(A) = 0$ if and only if $A$ is a terminal object;
\item\label{cor:4.13(2)} $\lambda(A) + \lambda(B) \le \lambda(A \myprod B)$ for every product $A \myprod B \in \mathrm{Ob}(\mathcal C)$ of $A$ by $B$.
\end{enumerate}
Then every $X \in \mathrm{Ob}(\mathcal C)$ is isomorphic to a finite product of \textsf{directly irreducible objects}, i.e., non-terminal objects of $\mathcal C$ each of which is not i\-so\-mor\-phic to a product of two non-terminal objects.
\end{corollary}

\begin{proof}
As noted in the comments above, the direct monoid $\mathcal V(\mathcal C)$ of isomorphism classes of $\mathcal C$ is reduced and commutative: Its identity is the isomorphism class of the terminal objects of $\mathcal C$. On the other hand, we have from conditions \ref{cor:4.13(1)} and \ref{cor:4.13(2)} that $\lambda(B) \le \lambda(A \myprod B)$ for all $A, B \in \mathrm{Ob}(\mathcal C)$ and every representative $A \myprod B \in \mathrm{Ob}(\mathcal C)$ of the product of $A$ by $B$, with equality if and only if $A$ is terminal. So, it is clear that $\mathcal V(\mathcal C)$ is unit-cancellative and, by Remark \ref{rem:3.10}\ref{rem:3.10(1)}, the divisibility preorder on $\mathcal V(\mathcal C)$ is artinian. Therefore, we get from Corollaries \ref{cor:4.1} and \ref{cor:4.4} that every isomorphism class of $\mathcal C$ factors as a (finite) product of atoms of $\mathcal V(\mathcal C)$. This finishes the proof, because an atom of $\mathcal V(\mathcal C)$ is, obviously, the isomorphism class of a directly irreducible object.
\end{proof}

Corollary \ref{cor:4.13} has many ``concrete realizations''. Below, we discuss one of them in detail: The focus will be on modules, but the same argument can be adapted to a whole variety of other objects for which a ``well-behaved'' notion of ``dimension'', ``rank'', etc., is available.

To begin, fix a (commutative or non-commutative) ring $R$. Following \cite[Definition (6.2) and Corollary (6.6)]{La99}, we let the \textsf{uniform dimension} $\dim_R(M)$ of a (left) $R$-module $M$ be the supremum of the set
\[
\{k \in \mathbb N^+ \colon N_1 \oplus_R \cdots \oplus_R N_k \text{ embeds into }M, \text{ for some non-zero } R\text{-submodules } N_1, \ldots, N_k\},
\] 
where $\oplus_R$ denotes a direct sum of $R$-modules and we take $\sup \emptyset := 0$. 
It is a fundamental fact that the u\-ni\-form dimension is \emph{additive}, in the sense that
\begin{equation}\label{equ:4.3(3)}
\dim_R(M \oplus_R N) = \dim_R(M) + \dim_R(N), 
\qquad 
\text{for all } R\text{-modules } M \text{ and } N,
\end{equation}
see \cite[Corollary (6.10), Part (1)]{La01}. This leads straight to the following:

\begin{corollary}\label{cor:4.14}
Let $R$ be a ring. Every $R$-module of finite uniform dimension is equal to a direct sum of finitely many indecomposable $R$-modules.
\end{corollary}

\begin{proof}
Let $\mathcal C$ be the full subcategory of the ordinary category $\mathsf{Mod}_R$ of $R$-modules and module homomorphisms whose objects are the $R$-modules with finite uniform dimension. It is a basic fact that $\mathsf{Mod}_R$ is a category with finite products: In particular, the terminal objects of $\mathsf{Mod}_R$ are the zero $R$-modules, and a canonical representative of the product of two $R$-modules $A$ and $B$ is their direct sum $A \oplus_R B$. Since the inclusion functor of $\mathcal C$ in $\mathsf{Mod}_R$  is fully faithful and, by \cite[Proposition 2.9.9]{Bor94}, fully faithful functors reflect limits, it follows by Eq.~\eqref{equ:4.3(3)} that $\mathcal C$, too, is a category with finite products (by additivity, the direct sum of two $R$-modules of finite uniform dimension is still an $R$-module of finite uniform dimension). 

On the other hand, if $\lambda$ is the function $\mathrm{Ob}(\mathcal C) \to \bf N$ that maps an $R$-module to its uniform dimension, then we also get from Eq.~\eqref{equ:4.3(3)} that $\lambda(A) + \lambda(B) \le \lambda(A \oplus_R B)$ for all $A, B \in \mathrm{Ob}(\mathcal C)$; moreover, $\lambda(A) = 0$ if and only if $A$ is a zero module (i.e., a terminal object of $\mathcal C$). Since $\lambda(A) = \lambda(B)$ when the $R$-mod\-ules $A$ and $B$ are isomorphic, we thus conclude from Corollary \ref{cor:4.13} (and the very definition of an indecomposable module) that every $R$-module is isomorphic and hence equal to a direct sum of indecomposables.
\end{proof}

By \cite[Corollary (6.7)(1)]{La99}, Corollary \ref{cor:4.14} generalizes the classical result that every artinian or noetherian module over a ring $R$ is an internal direct sum of indecomposable submodules. The corollary has a ``direct and simple'' proof all along the lines of the standard proof of the classical case, but the point here is rather that we obtained the result as an instance of an abstract ``object decomposition theorem'' (viz., Corollary \ref{cor:4.13}), in which we get to characterize the in\-decomposable $R$-modules as the atoms of a certain (reduced, unit-cancellative, commutative) monoid where the divisibility preorder is artinian. And by the same ``mechanical approach'', analogous conclusions can be made for other classes of objects.

\section{Closing remarks and open questions}
\label{sect:5}
Theorem \ref{thm:3.11} and its descendants work as a sort of black box for a variety of problems: The inputs of the black box are a monoid $H$ and an artinian preorder $\preceq$ on $H$; the output is the existence of certain factorizations for every ``large element'' of $H$, where an element is taken to be ``large'' if it is not $\preceq$-equivalent to the identity $1_H$ of $H$. In practice, if one's goal is to prove some kind of factorization theorem (as in the examples discussed in the previous sections), then the recipe set forth in this work consists of four steps (some of which are often trivial): 
\begin{enumerate}
	\item Build up a monoid $H$ that ``fits the factorization problem'' under consideration.
	\item Find a ``good candidate'' for the preorder $\preceq$.
	\item Prove that the preorder $\preceq$ is artinian.
	\item Characterize the $\preceq$-irreducibles of $H$. 
\end{enumerate}
One pro of the approach is that, similarly as with other top-down approaches, one can hope to bring ``factorization problems'' from ``distant areas'' under the umbrella of a unifying theory.

This said, there are many basic questions we could not answer. E.g., it follows from \cite[Theorem 4]{Gr51} that, if the ``divides from the left'' preorder and the ``divides from the right'' preorder are both noetherian, then also the divisibility preorder is noetherian. Moreover, we learned from Benjamin Steinberg on MathOverflow (see \texttt{\href{https://mathoverflow.net/questions/385422/}{https://mathoverflow.net/questions/385422/}}) that a monoid can satisfy both the ACCP and, say, the ACCPL without satisfying the ACCPR. However, we do not know whether a monoid $H$ satisfying both the ACCPR and the ACCPL does also satisfy the ACCP (cf.~Corollary \ref{cor:4.6} and Example \ref{exa:4.7}).
On a related note, is it true that every unit-cancellative monoid satisfying the ACCP is acyclic (cf.~Theorem \ref{thm:4.5})? Does an acyclic monoid $H$ defined by a presentation $\mathrm{Mon} \langle X \mid R \rangle$ with finitely many generators satisfy the ACCP? If not, what about a finite presentation? In Example \ref{exa:4.8}, we proved that the answer to the last question is negative if we drop the requirement that $H$ is acyclic and we ask in return that the monoid is reduced, cancellative, and atomic.

\section*{Acknowledgments}\label{subsec:acks}
I am indebted to Laura Cossu, Alfred Geroldinger, and Daniel Smertnig for many fruitful comments. Part of the paper was written in summer 2020, while I was on a research stay at the University of Graz, supported by the Austrian Science Fund FWF, Project No.~W1230. It was, however, in 2013 that the key ideas underlying the paper started growing, while I was still in France, moving back and forth between Paris 6, the \'Ecole Polytechnique, and Jean Monnet University in St-\'Etienne. Therefore, I also want to express my deepest gratitude to Alain Plagne and Fran\c{c}ois Hennecart for their support, at that time and afterwards. Last but not least, I am sincerely grateful to an anonymous referee for their incredibly careful reading of the manuscript and a wealth of constructive remarks.


\begin{thebibliography}{99}
%
\bibitem{Ad66} S.\,I.~Adian, \emph{Defining relations and algorithmic problems for groups and semigroups}, Trudy Mat.~Inst.~Steklov~\textbf{85} (1966), 3--123 (in Russian); Proc.~Steklov Inst.~Math.~\textbf{85} (1966), 1--152 (trans.~from the Russian by M.~Greendlinger).
%
\bibitem{AnVL96} D.\,D.~Anderson and S.~Valdes-Leon, \emph{Factorization in Commutative Rings with Zero Divisors}, Rocky Mountain J.~Math.~\textbf{26} (1996), No.~2, 439--480.
%
\bibitem{An-Tr18} A.\,A.~Antoniou and S.~Tringali, \emph{On the Arithmetic of Power Monoids and Sumsets in Cyclic Groups}, Pacific J.~Math.~\textbf{312} (2021), No.~2, 279--308.
%
\bibitem{BaGe14} N.\,R.~Baeth and A.~Geroldinger, \emph{Monoids of modules and arithmetic of direct-sum
decompositions}, Pacific J.~Math.~\textbf{271} (2014), No.~2, 257--319.
%
\bibitem{Ba-Sm20} N.\,R.~Baeth and D.~Smertnig, \emph{Lattices over Bass rings and graph agglomerations}, to appear in Algebr.~Represent.~Theory (\href{https://arxiv.org/abs/2006.10002}{\nolinkurl{arxiv.org/abs/2006.10002}}).
%
\bibitem{Ba-Sm15} N.\,R.~Baeth and D.~Smertnig, \emph{Factorization theory: From commutative to non-commutative settings}, J.~Algebra \textbf{441} (2015), 475--551.
%
\bibitem{BaWi13} N.\,R.~Baeth and R.~Wiegand, \emph{Factorization Theory and Decompositions of Modules}, Amer.~Math.~Monthly \textbf{120} (2013), No.~1, 3--34.
%
\bibitem{Bor94} F.~Borceux, \emph{Handbook of Categorical Algebra 1: Basic Category Theory},
Encycl.~Math.~Appl.~\textbf{50}, Cambridge Univ.~Press, 1994.
%
\bibitem{Co06} P.\,M.~Cohn, \emph{Free Ideal Rings and Localization in General Rings}, New Math.~Monogr.~\textbf{3}, Cambridge Univ.~Press, 2006.
%
\bibitem{Co67} P.\,M.~Cohn, \emph{Torsion modules over free ideal rings}, Proc.~London Math.~Soc., III.~Ser.~\textbf{17} (1967), 577--599.
%
\bibitem{Co64} P.\,M.~Cohn, \emph{Free ideal rings}, J.~Algebra \textbf{1} (1964), 47--69.
%
\bibitem{CoTr21}  L.~Cossu and S.~Tringali, \emph{Abstract Factorization Theorems with Applications to Idempotent Factorizations}, under review (\href{https://arxiv.org/abs/2108.12379}{\texttt{arXiv:2108.12379}})
%
\bibitem{CoZaZa18} L.~Cossu, P.~Zanardo, and U.~Zannier, \emph{Products of elementary matrices and non-Euclidean principal ideal domains}, J.~Algebra \textbf{501} (2018), 182--205.
%
\bibitem{Er68} J.\,A.~Erdos, \emph{On products of idempotent matrices}, Glasg.~Math.~J.~\textbf{8} (1967), 118--122.
%
\bibitem{Fa19} A.~Facchini, \emph{Semilocal Categories and Modules with Semilocal Endomorphism Rings}, Progr.~Math.~\textbf{331}, Birkh\"auser, 2019.
%
\bibitem{Fa02} A.~Facchini, \emph{Direct sum decomposition of modules, semilocal endomorphism rings, and Krull monoids}, J.~Algebra \textbf{256} (2002), No.~1, 280--307.
%
\bibitem{FaGeKaTr17} Y.~Fan, A.~Geroldinger, F.~Kainrath, and S.~Tringali, \emph{Arithmetic of commutative semigroups with a focus on semigroups of ideals and modules}, J.~Algebra Appl.~\textbf{16} (2017), No.~11, 42 pp.
%
\bibitem{Fa-Tr18} Y.~Fan and S.~Tringali, \emph{Power monoids: A bridge between Factorization Theory and Arithmetic Combinatorics}, J.~Algebra \textbf{512} (2018), 252--294.
%
\bibitem{Gal91} J.\,H.~Gallier, \emph{What's so special about Kruskal's theorem and the ordinal $\Gamma_0$? A survey of some results in proof theory}, Ann.~Pure Appl.~Logic \textbf{53} (1991), 199--260.
%
\bibitem{GeHK06} A.~Geroldinger and F.~Halter-Koch, \emph{Non-Unique Factorizations. Algebraic, Combinatorial and Analytic Theory}, Pure Appl.~Math.~\textbf{278}, Chapman \& Hall/CRC, 2006.
%
\bibitem{GeRe19} A.~Geroldinger and A.~Reinhart, \emph{The monotone catenary degree of monoids of ideals}, Internat.~J.~Algebra Comput.~\textbf{29} (2019), 419--457.
%
\bibitem{GeZh19} A.~Geroldinger and Q.~Zhong. \emph{Factorization theory in commutative monoids}, Semigroup Forum \textbf{100} (2020), 22--51.
%
\bibitem{Gr74} A.~Grams, \emph{Atomic rings and the ascending chain condition for principal ideals}, Math.~Proc.~Cambridge~Philos.~Soc.~\textbf{75} (1974), 321--329.
%
\bibitem{Gr51} J.\,A.~Green, \emph{On the Structure  of Semigroups}, Annals of Math.~\textbf{54} (1951) 163--172.
%
\bibitem{Ha06} P.~Harzheim, \emph{Ordered Sets}, Adv.~Math.~\textbf{7}, Springer, 2006.
%
\bibitem{Ho95} J.\,M.~Howie, \emph{Fundamentals of Semigroup Theory}, London Math.~Soc.~Monogr., New Ser.~\textbf{12}, Clarendon Press, 1995.
%
\bibitem{La99} T.\,Y.~Lam, \emph{Lectures on Modules and Rings}, Grad.~Texts in Math.~\textbf{189}, Springer, 1999.
%
\bibitem{La01} T.\,Y.~Lam, \emph{A First Course in Noncommutative Rings}, Grad.~Texts in Math.~\textbf{131}, Springer, 2001 (2nd edition).
%
\bibitem{MacLane98} S.~Mac Lane, \emph{Categories for the Working Mathematician}, Grad.~Texts in Math.~\textbf{5}, Springer, 1998 (2nd edition).
%
\bibitem{Ma95} G.~Mackiw, \emph{Permutations as Products of Transpositions}, Amer.~Math.~Monthly \textbf{102} (1995), No.~5, 438--440.
%
\bibitem{Ma-Zi-09} R.~Mazurek and M.~Ziembowski, \emph{The ascending chain condition for principal left or right ideals of skew generalized power series rings}, J.~Algebra \textbf{322} (2009), 983--994.
%
\bibitem{Rot06} J.\,J.~Rotman, \emph{A First Course in Abstract Algebra with Applications}, Prentice Hall, 2006 (3rd edition).
%
\bibitem{Tar38} A.~Tarski, \emph{\"Uber unerreichbare Kardinalzahlen}, Fund.~Math.~\textbf{30} (1938), 68--89.
%
\bibitem{WiWi09} R.~Wiegand and S.~Wiegand, ``Semigroups of modules: A survey'', pp.~335--349 in N.\,V.~Dung, F.~Guerriero, L.~Hammoudi, and P.~Kanwar (eds.), \emph{Rings, Modules and Representations}, Contemp.~Math.~\textbf{480}, Amer.~Math.~Soc., 2009.
%
\end{thebibliography}
\end{document}